\documentclass{amsart}%
\usepackage{amssymb}
\usepackage{amsfonts}%
\usepackage{amsmath}%
\usepackage{graphicx}
\providecommand{\U}[1]{\protect\rule{.1in}{.1in}}
\newtheorem{theorem}{Theorem}
\theoremstyle{plain}

\newtheorem{corollary}{Corollary}

\newtheorem{definition}{Definition}

\newtheorem{lemma}{Lemma}

\newtheorem{proposition}{Proposition}
\newtheorem{remark}{Remark}

\numberwithin{equation}{section}
\begin{document}
\title{Bilinear forms on the Dirichlet space}
\author{N. Arcozzi}
\address{Dipartimento do Matematica\\
Universita di Bologna\\
40127 Bologna, ITALY}
\author{R. Rochberg}
\address{Department of Mathematics\\
Washington University\\
St. Louis, MO 63130, U.S.A.}
\author{E. Sawyer}
\address{Department of Mathematics \& Statistics\\
McMaster University\\
Hamilton, Ontario, L8S 4K1, CANADA}
\author{B. D. Wick}
\address{Department of Mathematics\\
LeConte College\\
1523 Greene Street\\
University of South Carolina\\
Columbia, SC 29208, U.S.A.}
\thanks{N.A.'s work partially supported by the COFIN project Analisi Armonica, funded
by the Italian Minister for Research}
\thanks{R.R.'s work supported by the National Science Foundation under Grant No. 0070642}
\thanks{E.S.'s work supported by the National Science and Engineering Council of Canada.}
\thanks{B.D.W.'s work supported by the National Science Foundation under Grant No. 0752703}
\subjclass[2000]{31C25, 47B35, 30C85}
\maketitle

\section{Introduction}

\subsection{Overview}

Let $\mathcal{D}$ be the classical Dirichlet space, the Hilbert space of
holomorphic functions on the disk with inner product
\[
\left\langle f,g\right\rangle _{\mathcal{D}}=f(0)\overline{g(0)}%
+\int_{\mathbb{D}}f^{\prime}(z)\overline{g^{\prime}(z)}\text{ }dA
\]
and normed by $\left\Vert f\right\Vert _{\mathcal{D}}^{2}=\left\langle
f,f\right\rangle _{\mathcal{D}}.$ Given a holomorphic \textit{symbol function
}$b$ we define the associated Hankel type bilinear form, initially for
$f,g\in\mathcal{P}(\mathbb{D})$, the space of polynomials, by
\[
T_{b}\left(  f,g\right)  :=\left\langle fg,b\right\rangle _{\mathcal{D}}.
\]
The norm of $T_{b}$ is
\[
\left\Vert T_{b}\right\Vert _{\mathcal{D\times D}}:=\sup\left\{  \left\vert
T_{b}\left(  f,g\right)  \right\vert :\left\Vert f\right\Vert _{\mathcal{D}%
}=\left\Vert g\right\Vert _{\mathcal{D}}=1\right\}  .
\]
We say a positive measure $\mu$ on the disk is a \textit{Carleson measure for}
$\mathcal{D}$ if
\[
\left\Vert \mu\right\Vert _{CM(\mathcal{D})}:=\sup\left\{  \int_{\mathbb{D}%
}\left\vert f\right\vert ^{2}d\mu:\left\Vert f\right\Vert _{\mathcal{D}%
}=1\right\}  <\infty,
\]
and that a function $b$ is in the space $\mathcal{X}$ if the measure $d\mu
_{b}:=\left\vert b^{\prime}(z)\right\vert ^{2}dA\ $is a Carleson measure. We
norm $\mathcal{X}$ by%
\[
\left\Vert b\right\Vert _{\mathcal{X}}:=\left\vert b(0)\right\vert +\left\Vert
\left\vert b^{\prime}(z)\right\vert ^{2}dA\right\Vert _{CM(\mathcal{D})}^{1/2}%
\]
and denote by $\mathcal{X}_{0}$ the norm closure in $\mathcal{X}$ of the space
of polynomials.

Our main result is

\begin{theorem}
\ \label{main}

\begin{enumerate}
\item $T_{b}$ is bounded if and only if $b\in\mathcal{X}.$ In that case
\[
\left\Vert T_{b}\right\Vert _{\mathcal{D\times D}}\approx\left\Vert
b\right\Vert _{\mathcal{X}}.
\]

\item $T_{b}$ is compact if and only if $b\in\mathcal{X}_{0}.$
\end{enumerate}
\end{theorem}

This result, which had been conjectured by the second author for some time, is
part of an intriguing pattern of results involving boundedness of Hankel forms
on Hardy spaces in one and several variables and boundedness of
Schr\"{o}dinger operators on the Sobolev space. We recall those results in the
next subsection.

Boundedness criteria for bilinear forms can be recast as weak factorization of
function spaces and we discuss that in the third subsection. The first
statement in Theorem \ref{main} is equivalent to a weak factorization of the
predual of $\mathcal{X};$ in notation we introduce there
\begin{equation}
\left(  \mathcal{D}\odot\mathcal{D}\right)  ^{\ast}=\mathcal{X}%
.\label{weakdirichlet}%
\end{equation}
In the final subsection we describe the relation between Theorem \ref{main}
and classical results about Hankel matrices.

The proof of Theorem \ref{main} is in Sections 2 and 3. It is easy to see that
$\left\Vert T_{b}\right\Vert _{\mathcal{D\times D}}\leq C\left\Vert
b\right\Vert _{\mathcal{X}}.$ To obtain the other inequality we must use the
boundedness of $T_{b}$ to show $\left\vert b^{\prime}\right\vert ^{2}dA$ is a
Carleson measure. Analysis of the capacity theoretic characterization of
Carleson measures due to Stegenga allows us to focus attention on a certain
set $V$ in $\mathbb{D}$ and the relative sizes of $\int_{V}\left\vert
b^{\prime}\right\vert ^{2}$ and the capacity of the set $\bar{V}\cap
\partial\mathbb{\bar{D}}.$ To compare these quantities we construct
$V_{\text{exp}},$ an expanded version of the set $V$ which satisfies two
conflicting conditions. First, $V_{\text{exp}}$ is not much larger than $V$,
either when measured by $\int_{V_{\text{exp}}}\left\vert b^{\prime}\right\vert
^{2}$ or by the capacity of the $\overline{V_{\text{exp}}}\cap\partial
\mathbb{\bar{D}}.$ Second, $\mathbb{D}\setminus V_{\text{exp}}$ is well
separated from $V$ in a way that allows the interaction of quantities
supported on the two sets to be controlled. Once this is done we can construct
a function $\Phi_{V}\in\mathcal{D}$ which is approximately one on $V$ and
which has $\Phi_{V}^{\prime}$ approximately supported on $\mathbb{D}\setminus
V_{\text{exp}}.$ Using $\Phi_{V}$ we build functions $f$ and $g$ with the
property that
\[
\left\vert T_{b}(f,g)\right\vert =\int_{V}\left\vert b^{\prime}\right\vert
^{2}+\text{ error.}%
\]
The technical estimates on $\Phi_{V}$ allow us to show that the error term is
small and the boundedness of $T_{b}$ then gives the required control of
$\int_{V}\left\vert b^{\prime}\right\vert ^{2}$.

Once the first part of the theorem is established, the second follows rather directly.

\subsection{Other Bilinear Forms}

The Hardy space of the unit disk, $H^{2}\left(  \mathbb{D}\right)  ,$ can be
defined as the space of holomorphic functions on the disk with inner product
\[
\left\langle f,g\right\rangle _{H^{2}\left(  \mathbb{D}\right)  }%
=f(0)\overline{g(0)}+\int_{\mathbb{D}}f^{\prime}(z)\overline{g^{\prime}%
(z)}\text{ }(1-\left\vert z\right\vert ^{2})dA
\]
and normed by $\left\Vert f\right\Vert _{H^{2}\left(  \mathbb{D}\right)  }%
^{2}=\left\langle f,f\right\rangle _{H^{2}\left(  \mathbb{D}\right)  }.$ Given
a holomorphic \textit{symbol function }$b$ the Hankel form with symbol $b$ is
the bilinear form
\begin{equation}
T_{b}^{H^{2}\left(  \mathbb{D}\right)  }\left(  f,g\right)  :=\left\langle
fg,b\right\rangle _{H^{2}\left(  \mathbb{D}\right)  }.\label{hardy}%
\end{equation}
The boundedness criteria for such forms was given by Nehari in 1957 \cite{N}.
He used the fact that functions in the Hardy space $H^{1}$ can be written as
the product of functions in $H^{2}$ and showed $T_{b}^{H^{2}\left(
\mathbb{D}\right)  }$ will be bounded if and only if $b$ is in the dual space
of $H^{1}.$ Using Ch. Fefferman's identification of the dual of $H^{1}$ we can
reformulate this in the language of Carleson measures. We say a positive
measure $\mu$ on the disk is a \textit{Carleson measure for} $H^{2}\left(
\mathbb{D}\right)  $ if
\[
\left\Vert \mu\right\Vert _{CM(H^{2}\left(  \mathbb{D}\right)  )}%
:=\sup\left\{  \int_{\mathbb{D}}\left\vert f\right\vert ^{2}d\mu:\left\Vert
f\right\Vert _{H^{2}\left(  \mathbb{D}\right)  }=1\right\}  <\infty.
\]
The form $T_{b}^{H^{2}\left(  \mathbb{D}\right)  }$ is bounded if and only if
$b$ is in the function space $BMO$ or, equivalently, if and only if
\[
\left\vert b^{\prime}(z)\right\vert ^{2}(1-\left\vert z\right\vert ^{2})dA\in
CM(H^{2}\left(  \mathbb{D}\right)  ).
\]

Later, in \cite{CRW}, Nehari's theorem was viewed as a result about
Calder\'{o}n-Zygmund singular integrals on spaces of homogenous type and an
analogous result was proved for $H^{2}\left(  \partial\mathbb{B}^{n}\right)  ,
$ the Hardy space of the sphere in complex $n-$space. In that context the
Hankel form is defined similarly
\[
T_{b}^{H^{2}\left(  \partial\mathbb{B}^{n}\right)  }\left(  f,g\right)
:=\left\langle fg,b\right\rangle _{H^{2}\left(  \partial\mathbb{B}^{n}\right)
}.
\]
That form is bounded if and only if $b$ is in $BMO\left(  \partial
\mathbb{B}^{n}\right)  $ or, equivalently, if and only if, with $\nabla$
denoting the invariant gradient on the ball,%
\[
\left\vert \nabla b(z)\right\vert ^{2}dV\in CM\left(  H^{2}\left(
\partial\mathbb{B}^{n}\right)  \right)  .
\]

The approach in \cite{CRW} is not well suited for analysis on the Hardy space
of the polydisk, $H^{2}\left(  \mathbb{D}^{n}\right)  .$ However Ferguson,
Lacey, and Terwilleger were able to extend methods of multivariable harmonic
analysis and obtain a result for $H^{2}\left(  \mathbb{D}^{n}\right)  $
\cite{FL}, \cite{LT}. They showed that a Hankel form on $H^{2}\left(
\mathbb{D}^{n}\right)  ,$ again defined as a form whose value only depends on
the product of its arguments, is bounded if and only if the symbol function
$b$ lies in $BMO\left(  \mathbb{D}^{n}\right)  $ or, equivalently, if and only
if derivatives of $b$ can be used to generate a Carleson measure for
$H^{2}\left(  \mathbb{D}^{n}\right)  .$

In \cite{MV} Maz'ya and Ververbitsky presented a boundedness criterion for a
bilinear form associated to the Schr\"{o}dinger operator. Although their
viewpoint and proof techniques were quite different from those used for Hankel
forms, their result is formally very similar. We change their formulation
slightly to make the analogy more visible, our $b$ is related to their $V$ by
$b=-\Delta^{-1}V$. Let $\mathring{L}_{2}^{1}\left(
\mathbb{R}
^{n}\right)  \mathbb{\ }$be the energy space (homogenous Sobolev space)
obtained by completing $C_{0}^{\infty}\left(
\mathbb{R}
^{n}\right)  $ with respect to the quasinorm induced by the Dirichlet inner
product%
\[
\left\langle f,g\right\rangle _{\text{Dir}}=\int_{%
\mathbb{R}
^{n}}\nabla f\cdot\overline{\nabla g}\text{ }dx.
\]
Given $b$, a bilinear Schr\"{o}dinger form on $\mathring{L}_{2}^{1}\left(
\mathbb{R}
^{n}\right)  \times\mathring{L}_{2}^{1}\left(
\mathbb{R}
^{n}\right)  $ is defined by%
\[
S_{b}(f,g)=\left\langle fg,b\right\rangle _{\text{Dir}}.
\]
We will say a measure $\mu$ on $%
\mathbb{R}
^{n}$ is a \textit{Carleson measure for the energy space} if
\[
\left\Vert \mu\right\Vert _{CM(\mathring{L}_{2}^{1}\left(
\mathbb{R}
^{n}\right)  )}:=\sup\left\{  \int_{\mathbb{%
\mathbb{R}
}^{n}}\left\vert f\right\vert ^{2}d\mu:\left\Vert f\right\Vert _{\mathring
{L}_{2}^{1}\left(
\mathbb{R}
^{n}\right)  }=1\right\}  <\infty.
\]
Corollary 2 of \cite{MV} is that $S_{b}$ is bounded if and only if
\[
|\left(  -\Delta\right)  ^{1/2}b|^{2}dx\in CM(\mathring{L}_{2}^{1}\left(
\mathbb{R}
^{n}\right)  ).
\]

It would be very satisfying to know an underlying reason for the similarity of
these various results to each other and to Theorem 1.

\subsection{Reformulation in Terms of Weak Factorization}

In his proof Nehari used the fact that any function $f\in H^{1}\left(
\mathbb{D}\right)  $ could be factored as $f=gh$ with $g,h\in H^{2}\left(
\mathbb{D}\right)  ,$ $\left\Vert f\right\Vert _{H^{1}\left(  \mathbb{D}%
\right)  }=\left\Vert g\right\Vert _{H^{2}\left(  \mathbb{D}\right)
}\left\Vert h\right\Vert _{H^{2}\left(  \mathbb{D}\right)  }.$ In \cite{CRW}
the authors develop a weak substitute for this. For two Banach spaces of
functions, $\mathcal{A}$ and $\mathcal{B}$, defined on the same domain, define
the weakly factored space $\mathcal{A\odot B}$ to be the completion of finite
sums $f=\sum a_{i}b_{i};$ $\left\{  a_{i}\right\}  \subset\mathcal{A}$,
$\left\{  b_{i}\right\}  \subset\mathcal{B}$ using the norm
\[
\left\Vert f\right\Vert _{\mathcal{A}\odot\mathcal{B}}=\inf\left\{
\sum\left\Vert a_{i}\right\Vert _{\mathcal{A}}\left\Vert b_{i}\right\Vert
_{\mathcal{B}}:f=\sum a_{i}b_{i}\right\}  .
\]

It is shown in \cite{CRW} that $H^{2}\left(  \partial\mathbb{B}^{n}\right)
\mathcal{\odot}H^{2}\left(  \partial\mathbb{B}^{n}\right)  =H^{1}\left(
\partial\mathbb{B}^{n}\right)  $ and consequentially
\begin{equation}
\left(  H^{2}\left(  \partial\mathbb{B}^{n}\right)  \mathcal{\odot}%
H^{2}\left(  \partial\mathbb{B}^{n}\right)  \right)  ^{\ast}=BMO\left(
\partial\mathbb{B}^{n}\right)  .\label{weakhardy}%
\end{equation}
(In this context, by "$=$" we mean equality of the function spaces and
equivalence of the norms.) Based on the analogy between (\ref{weakdirichlet})
and (\ref{weakhardy}) we think of $\mathcal{D}\odot\mathcal{D}$ as a type of
$H^{1}$ space and of $\mathcal{X}$ as a type of $BMO$ space. That viewpoint is
developed further in \cite{ARSW}.

The precise formulation of (\ref{weakdirichlet}) is the following corollary.

\begin{corollary}
For $b\in\mathcal{X}$ set $\Lambda_{b}h=T_{b}\left(  h,1\right)  ,$ then
$\Lambda_{b}\in\left(  \mathcal{D}\odot\mathcal{D}\right)  ^{\ast}.$
Conversely, if $\Lambda\in\left(  \mathcal{D}\odot\mathcal{D}\right)  ^{\ast}$
there is a unique $b\in\mathcal{X}$ so that for all $h\in\mathcal{P}%
(\mathbb{D})$ we have $\Lambda h=T_{b}\left(  h,1\right)  =\Lambda_{b}h.$ In
both cases $\left\Vert \Lambda_{b}\right\Vert _{\left(  \mathcal{D}%
\odot\mathcal{D}\right)  ^{\ast}}\approx\left\Vert b\right\Vert _{\mathcal{X}%
}$.

\begin{proof}
If $b\in\mathcal{X}$ and $f\in\mathcal{D}\odot\mathcal{D}$, say $f=\sum
g_{i}h_{i}$ with $\sum\left\Vert g_{i}\right\Vert _{\mathcal{D}}\left\Vert
h_{i}\right\Vert _{\mathcal{D}}\leq\left\Vert f\right\Vert _{\mathcal{D}%
\odot\mathcal{D}}+\varepsilon$, then
\begin{align*}
\left\vert \Lambda_{b}f\right\vert  & =\left\vert \sum_{i=1}^{\infty
}\left\langle g_{i}h_{i},b\right\rangle _{\mathcal{D}}\right\vert =\left\vert
\sum_{i=1}^{\infty}T_{b}(g_{i},h_{i})\right\vert \\
& \leq\left\Vert T_{b}\right\Vert \sum_{i=1}^{\infty}\left\Vert g_{i}%
\right\Vert _{\mathcal{D}}\left\Vert h_{i}\right\Vert _{\mathcal{D}}%
\leq\left\Vert T_{b}\right\Vert \left(  \left\Vert f\right\Vert _{\mathcal{D}%
\odot\mathcal{D}}+\varepsilon\right)  .
\end{align*}
It follows that $\Lambda_{b}f=\left\langle f,b\right\rangle _{\mathcal{D}}$
defines a continuous linear functional on $\mathcal{D}\odot\mathcal{D}$ with
$\left\Vert \Lambda_{b}\right\Vert \leq\left\Vert T_{b}\right\Vert $.

Conversely, if $\Lambda\in\left(  \mathcal{D}\odot\mathcal{D}\right)  ^{\ast}%
$with norm $\left\Vert \Lambda\right\Vert $, then for all $f\in\mathcal{D}$
\[
\left\vert \Lambda f\right\vert =\left\vert \Lambda\left(  f\cdot1\right)
\right\vert \leq\left\Vert \Lambda\right\Vert \left\Vert f\right\Vert
_{\mathcal{D}}\left\Vert 1\right\Vert _{\mathcal{D}}=\left\Vert \Lambda
\right\Vert \left\Vert f\right\Vert _{\mathcal{D}}.
\]
Hence there is a unique $b\in\mathcal{D}$ such that $\Lambda f=\Lambda_{b}f $
for $f\in\mathcal{D}$. Finally, if $f=gh$ with $g,h\in\mathcal{D}$ we have%
\begin{align*}
\left\vert T_{b}\left(  g,h\right)  \right\vert  & =\left\vert \left\langle
gh,b\right\rangle _{\mathcal{D}}\right\vert =\left\vert \Lambda_{b}%
f\right\vert =\left\vert \Lambda f\right\vert \\
& \leq\left\Vert \Lambda\right\Vert \left\Vert f\right\Vert _{\mathcal{D}%
\odot\mathcal{D}}\leq\left\Vert \Lambda\right\Vert \left\Vert g\right\Vert
_{\mathcal{D}}\left\Vert h\right\Vert _{\mathcal{D}},
\end{align*}
which shows that $T_{b}$ extends to a continuous bilinear form on
$\mathcal{D}\odot\mathcal{D}$ with $\left\Vert T_{b}\right\Vert \leq\left\Vert
\Lambda\right\Vert $. By Theorem \ref{main} we conclude $b\in\mathcal{X}$ and
collecting the estimates that $\left\Vert \Lambda\right\Vert =\left\Vert
\Lambda_{b}\right\Vert _{\left(  \mathcal{D}\odot\mathcal{D}\right)  ^{\ast}%
}\approx\left\Vert T_{b}\right\Vert \approx\left\Vert b\right\Vert
_{\mathcal{X}}.$
\end{proof}
\end{corollary}

Define the space $\partial^{-1}\left(  \partial\mathcal{D}\odot\mathcal{D}%
\right)  $ to be the completion of the space of functions $f$ which have
$f^{\prime}=\sum_{i=1}^{N}g_{i}^{\prime}h_{i}$ (and thus $f=\partial^{-1}%
\sum\left(  \partial g_{i}\right)  h_{i})$) using the norm
\[
\left\Vert f\right\Vert _{\partial^{-1}\left(  \partial\mathcal{D}%
\odot\mathcal{D}\right)  }=\inf\left\{  \sum\left\Vert g_{i}\right\Vert
_{\mathcal{D}}\left\Vert h_{i}\right\Vert _{\mathcal{D}}:f^{\prime}=\sum
_{i=1}^{N}g_{i}^{\prime}h_{i}\right\}  .
\]
Using the previous corollary we can recapture, but by a very indirect route,
an earlier result of Coifman-Murai \cite{CM}, Tolokonnikov \cite{T}, and
Rochberg-Wu \cite{RW}.

\begin{corollary}
[\cite{CM}, \cite{T}, \cite{RW}]%
\[
\left(  \partial^{-1}\left(  \partial\mathcal{D}\odot\mathcal{D}\right)
\right)  ^{\ast}=\mathcal{X}\text{.}%
\]

\begin{proof}
As in the previous proof this statement is equivalent to a boundedness
criterion for a class of bilinear forms. In this case the forms of interest
are those defined on $\mathcal{D}\times\mathcal{D}$ by%
\[
K_{b}(f,g)=\int_{\mathbb{D}}f^{\prime}g\overline{b^{\prime}}dV.
\]
The proof given later that $T_{b}$ is bounded if $b\in\mathcal{X}$ in fact
shows that $K_{b}$ is bounded and then notes that
\begin{equation}
T_{b}(f,g)=K_{b}(f,g)+K_{b}(g,f)+(fg\bar{b})(0).\label{summ}%
\end{equation}
In the other direction, if $K_{b}$ is bounded then the same relation shows
$T_{b}$ is bounded and we can then appeal to Theorem \ref{main}.
\end{proof}
\end{corollary}

The proofs in \cite{CM}, \cite{T}, and \cite{RW} give, explicitly or
implicitly, estimates from below for $\left\vert K_{b}(f,g)\right\vert .$ In
proving Theorem \ref{main} we need to estimate $\left\vert T_{b}%
(f,g)\right\vert $ from below. We avoided using the representation
(\ref{summ}) as a starting point because it was unclear how to analyze the
potential cancellation between terms on the right hand side of (\ref{summ}).

Combining the previous corollaries we have, with the obvious notation,

\begin{corollary}%
\[
\partial\left(  \mathcal{D}\odot\mathcal{D}\right)  =\partial\mathcal{D}%
\odot\mathcal{D}\text{.}%
\]

\end{corollary}

In contrast
\[
\partial\left(  \mathcal{D}\odot\mathcal{D}\right)  \neq\partial
^{1/2}\mathcal{D}\odot\partial^{1/2}\mathcal{D}\text{.}%
\]
To see this note that $\partial^{1/2}\mathcal{D}\odot\partial^{1/2}%
\mathcal{D}=H^{2}\left(  \mathbb{D}\right)  \mathcal{\odot}H^{2}\left(
\mathbb{D}\right)  =H^{1}\left(  \mathbb{D}\right)  $ and that $f(z)=(\log
\left(  1-z\right)  )^{3/2}$ satisfies $f^{\prime}\in$ $\partial\left(
\mathcal{D}\odot\mathcal{D}\right)  ,$ $f^{\prime}\notin H^{1}.$

\subsection{Reformulation in Terms of Matrices.}

If $T_{b}$ is given by (\ref{hardy}) with $b(z)=\sum b_{n}z^{n}$ then the
matrix representation of $T_{b}$ with respect to the monomial basis is
$\left(  \bar{b}_{i+j}\right)  $. Nehari's theorem gives a boundedness
condition for such Hankel matrices; matrices $\left(  a_{i,j}\right)  $ for
which $a_{i,j}$ is a function of $i+j$. There are analogous results for Hankel
forms on Bergman spaces. Those forms have matrices
\begin{equation}
\left(  \left(  i+1\right)  ^{\alpha}\left(  j+1\right)  ^{\beta}\left(
i+j+1\right)  ^{\gamma}\bar{b}(i+j)\right) \label{A}%
\end{equation}
with $\alpha,\beta>0$ and are bounded if and only if $b(z)$ is in the Bloch
space. The criteria for (\ref{A}) to belong to the Schatten--von Neumann
classes is known if $\min\left\{  \alpha,\beta\right\}  >-1/2$ and it is known
that those results do not extend to $\min\left\{  \alpha,\beta\right\}
\leq-1/2.$ For all of this see \cite[Ch 6.8]{P}.

The matrix representations of the forms $T_{b}$ and $K_{b}$ with respect the
the basis of normalized monomials of $\mathcal{D}$ are of the form (\ref{A})
with $\left(  \alpha,\beta\right)  $ equal to $\left(  -1/2,-1/2\right)  $ in
the first case and $\left(  -1/2,1/2\right)  $ in the second.

\section{Preliminary Steps in the Proof}

\subsection{The Proof of (2) Given (1)}

Suppose $T_{b}$ is compact. For any holomorphic function $k(z)$ on
$\mathbb{D}$ and $r,$ $0<\,r<1,$ set $S_{r}k(z)=k(rz).$ A computation with
monomials verifies that
\[
T_{S_{r}b}\left(  f,g\right)  =T_{b}(S_{r}f,S_{r}g).
\]
As $r\rightarrow1,$ $S_{r}$ converges strongly to $I.$ Using this and the fact
that $T_{b}$ is compact we obtain $\lim\left\Vert T_{S_{r}b}-T_{b}\right\Vert
=0.$ Hence, by the first part of the theorem $\lim\left\Vert S_{r}%
b-b\right\Vert _{\mathcal{X}}=0.$ The Taylor coefficients of $S_{r}b$ decay
geometrically, hence $S_{r}b\in\mathcal{X}_{0}$ and thus $b\in\mathcal{X}_{0}$.

In the other direction note that if $b$ is a polynomial then $T_{b}$ is finite
rank and hence compact. If $\left\{  b_{n}\right\}  \subset\mathcal{P}%
(\mathbb{D)}$ is a sequence of polynomials which converge in norm to
$b\in\mathcal{X}_{0}$ then, by the first part of the theorem $T_{b}$ is the
norm limit of the $T_{b_{n}}$ and hence is also compact.

\subsection{The Proof of The Easy Direction of (1)}

Suppose that $\mu_{b}$ is a $\mathcal{D}$-Carleson measure. For $f,g\in
\mathcal{P}\left(  \mathbb{D}\right)  $ we have
\begin{align*}
\left\vert T_{b}\left(  f,g\right)  \right\vert  & =\left\vert f\left(
0\right)  g\left(  0\right)  \overline{b\left(  0\right)  }+\int
\limits_{\mathbb{D}}\left(  f^{\prime}\left(  z\right)  g\left(  z\right)
+f\left(  z\right)  g^{\prime}\left(  z\right)  \right)  \overline{b^{\prime
}\left(  z\right)  }dA\right\vert \\
& \leq\left\vert f\left(  0\right)  g\left(  0\right)  b\left(  0\right)
\right\vert +\int\limits_{\mathbb{D}}\left\vert f^{\prime}\left(  z\right)
g\left(  z\right)  b^{\prime}\left(  z\right)  \right\vert dA+\int
\limits_{\mathbb{D}}\left\vert f\left(  z\right)  g^{\prime}\left(  z\right)
b^{\prime}\left(  z\right)  \right\vert dA\\
& \leq\left\vert (fgb)(0)\right\vert +\left\Vert f\right\Vert _{\mathcal{D}%
}\left(  \int\nolimits_{\mathbb{D}}\left\vert g\right\vert ^{2}d\mu
_{b}\right)  ^{1/2}+\left\Vert g\right\Vert _{\mathcal{D}}\left(
\int\nolimits_{\mathbb{D}}\left\vert f\right\vert ^{2}d\mu_{b}\right)
^{1/2}\\
& \leq C\left(  \left\vert b\left(  0\right)  \right\vert +\left\Vert \mu
_{b}\right\Vert _{\mathcal{D}-Carleson}\right)  \left\Vert f\right\Vert
_{\mathcal{D}}\left\Vert g\right\Vert _{\mathcal{D}}\\
& =C\left\Vert b\right\Vert _{\mathcal{X}}\left\Vert f\right\Vert
_{\mathcal{D}}\left\Vert g\right\Vert _{\mathcal{D}}.
\end{align*}
Thus $T_{b}$ has a bounded extension to $\mathcal{D}\times\mathcal{D}$ with
$\left\Vert T_{b}\right\Vert \leq C\left\Vert b\right\Vert _{\mathcal{X}}.$

We note for later that if $T_{b}$ extends to a bounded bilinear form on
$\mathcal{D}$ then $b\in\mathcal{D}$, equivalently, $d\mu_{b}$ is a finite
measure. To see this note that for all $f\in\mathcal{P}\left(  \mathbb{D}%
\right)  ,$ $\left\vert \left\langle f,b\right\rangle _{\mathcal{D}%
}\right\vert =\left\vert T_{b}\left(  f,1\right)  \right\vert \leq\left\Vert
T_{b}\right\Vert \left\Vert f\right\Vert _{\mathcal{D}}\left\Vert 1\right\Vert
_{\mathcal{D}}$.\ Thus $b\in\mathcal{D}$ and
\begin{equation}
\left\Vert b\right\Vert _{\mathcal{D}}\leq C\left\Vert T_{b}\right\Vert
.\label{Dnorm}%
\end{equation}

\subsection{Disk Capacity and Disk Blowups}

To complete the proof of Theorem \ref{main} we must show that if $T_{b}$ is
bounded then $\mu_{b:}=\left\vert b^{\prime}\right\vert ^{2}dA$ is a
$\mathcal{D}$-Carleson measure. We will do this by showing that $\mu_{b}$
satisfies a capacitary condition introduced by Stegenga \cite{St}.

For an interval $I$ in the circle we let $I_{m}$ be its midpoint and
$z(I)=\left(  1-\left\vert I\right\vert /2\pi\right)  I_{m}$ be the associated
index point in the disk. In the other direction let $I(z)$ be the interval
such that $z(I(z))=z.$ Let $T(I)$ be the tent over $I,$ the convex hull of $I
$ and $z(I)$ and let $T\left(  z\right)  =T\left(  z\left(  I\right)  \right)
:=T\left(  I\right)  $. More generally, for any open subset $H$ of the circle
$\mathbb{T}$, we define $T\left(  H\right)  ,$ the \emph{tent region} of $H$
in the disk $\mathbb{D}$, by%
\[
T\left(  H\right)  =\bigcup\limits_{I\subset H}T\left(  I\right)  .
\]

For $G$ in the circle $\mathbb{T}$ define the capacity of $G$ by
\begin{equation}
Cap_{\mathbb{D}}G=\inf\left\{  \left\Vert \psi\right\Vert _{\mathcal{D}}%
^{2}:\psi\left(  0\right)  =0,\operatorname{Re}\psi\left(  z\right)
\geq1\text{ for }z\in G\right\}  .\label{extremal}%
\end{equation}
Stegenga \cite{St} has shown that $\mu$ is a $\mathcal{D}$-Carleson measure
exactly if for any finite collection of disjoint arcs $\left\{  I_{j}\right\}
_{j=1}^{N}$ in the circle $\mathbb{T}$ we have
\begin{equation}
\mu\left(  \bigcup\nolimits_{j=1}^{N}T\left(  I_{j}\right)  \right)  \leq
C\ Cap_{\mathbb{D}}\left(  \bigcup\nolimits_{j=1}^{N}I_{j}\right)
.\label{Steg}%
\end{equation}

We will need to understand how the capacity of a set changes if we expand it
in certain ways. For $I$ an open arc and $0<\rho\leq1$, let $I^{\rho}$ be the
arc concentric with $I$ having length $\left\vert I\right\vert ^{\rho}$.

\begin{definition}
[disk blowup]\label{diskblowup}For $G$ open in $\mathbb{T}$ we call%
\[
G_{\mathbb{D}}^{\rho}=\bigcup_{I\subset G}T\left(  I^{\rho}\right)
\]
the \emph{disk blowup} (of order $\rho$) of $G$.
\end{definition}

The important feature of the disk blowup is that it achieves a good geometric
separation between $\mathbb{D}\setminus G_{\mathbb{D}}^{\rho}$ and
$G_{\mathbb{D}}^{1}=T\left(  G\right)  .$ This plays a crucial role in using
Schur's test to estimate an integral later, as well as in estimating an error
term near the end of the paper.

\begin{lemma}
\label{geoseparation}Let $G$ be an open subset of the circle $\mathbb{T}$. If
$w\in G_{\mathbb{D}}^{1}=T\left(  G\right)  $ and $z\notin G_{\mathbb{D}%
}^{\rho}$ then $\left\vert z-w\right\vert \geq(1-\left\vert w\right\vert
^{2})^{\rho}$.

\begin{proof}
The inequality follows from the definition of $G_{\mathbb{D}}^{\rho}$ and the
inclusion $T\left(  I^{\rho}\right)  \subset\left\{  z:\left\vert z-z\left(
I\right)  \right\vert <2(1-\left\vert z\left(  I\right)  \right\vert )^{2\rho
}\right\}  .$
\end{proof}
\end{lemma}

It would be useful to us if we knew there were constants $C_{\rho},$
$0<\rho<1,$ such that%
\begin{equation}
Cap_{\mathbb{D}}\left(  \bigcup_{I\subset G}I^{\rho}\right)  \leq C_{\rho
}Cap_{\mathbb{D}}G.\label{B}%
\end{equation}
and%
\begin{equation}
\lim_{\rho\rightarrow1^{-}}C_{\rho}=1.\label{lim}%
\end{equation}
Bishop proved (\ref{B}) \cite{Bis} but did not obtain (\ref{lim}) and we could
not obtain it directly. In the next subsection we obtain Lemma \ref{contain},
an analog of (\ref{B}) and (\ref{lim}) in a tree model, and that will play an
important role in the proof. After we show that tree and disk are comparible,
Corollary \ref{capequiv}, then we will also have (\ref{lim}).

\subsection{Tree Capacity and Tree Blowups}

In our study of capacities and approximate extremals it will sometimes be
convenient to transfer our arguments to and from the Bergman tree
$\mathcal{T}$ and to work with the associated tree capacities. We now recall
the notation associated to $\mathcal{T}.$ Further properties of $\mathcal{T}$
are in the Appendix and a more extensive investigation with other applications
is in \cite{ArRoSa4}.

Let $\mathcal{T}$ be the standard Bergman tree in the unit disk $\mathbb{D}$.
That is $\mathcal{T=}\left\{  x\right\}  $ is the index set for the subsets
$\left\{  B_{x}\right\}  $ of $\mathbb{D}$ obtained by decomposing
$\mathbb{D},$ first with the circles $C_{k}=\left\{  z:\left\vert z\right\vert
=1-2^{-k}\right\}  ,$ $k=1,2,...$ and then for each $k$ making $2^{k}$ radial
cuts in the ring bounded by $C_{k}$ and $C_{k+1}.$ We refer to the $\left\{
B_{x}\right\}  $ as boxes and we emphasize the standard bijection between the
boxes and the intervals on the circle $\left\{  I(B_{x})\right\}  $ obtained
by radial projection of the boxes. This also induces a bijection with the
point set $\left\{  z(I(B_{x}))\right\}  $ in the disk$,$ furthermore
$z(I(B_{x}))\in B_{x}$. At times we will use the label $x$ to denote the point
$z(I(B_{x})).$

$\mathcal{T}$ is a rooted dyadic tree with root $\left\{  0\right\}  $ which
we denote $o.$ For a vertex $x$ of $\mathcal{T}$ we denote its immediate
predecessor by $x^{-1}$ and its two immediate successors by $x_{+}$ and
$x_{-}.$ We let $d(x)$ equal the number of nodes on the geodesic $[o,x]$. The
successor set of $x$ is $S(x)=\left\{  y\in\mathcal{T}:y\geq x\right\}  .$

We say that $S\subset\mathcal{T}$ is a \emph{stopping time} if no pair of
distinct points in $S$ are comparable in $\mathcal{T}$. Given stopping times
$E,F\subset\mathcal{T}$ we say that $F\succ E$ if for every $x\in F$ there is
$y\in E$ above $x,$ i.e., with $x>y$. For stopping times $F\succ E$ denote by
$\mathcal{G}\left(  E,F\right)  $ the union of all those geodesics connecting
a point of $x\in F$ to the point $y\in E$ above it.

The bijections between $\left\{  B_{x}\right\}  $, $\left\{  I(B_{x})\right\}
$, and $\left\{  z(I(B_{x}))\right\}  $ induce bijections between other sets.
We will be particularly interested in three types of sets:

\begin{itemize}
\item \emph{stopping times} $W$ in the tree $\mathcal{T}$;

\item $\mathcal{T}$-\emph{open subsets} $G$ of the circle $\mathbb{T}$;

\item $\mathcal{T}$-\emph{tent regions} $\Gamma$ of the disk $\mathbb{D}$.
\end{itemize}

The bijections are given as follows. For $W$ a \emph{stopping time} in
$\mathcal{T}$, its associated $\mathcal{T}$-\emph{open set} in $\mathbb{T}$ is
the $\mathcal{T}$-shadow $S_{\mathcal{T}}\left(  W\right)  =\cup\left\{
I(x):x\in W\right\}  $ of $W$ on the circle (this also \emph{defines} the
collection of $\mathcal{T}$-open sets). The associated $\mathcal{T}%
$-\emph{tent region} in $\mathbb{D}$ is $T_{\mathcal{T}}\left(  W\right)
=\cup\left\{  T\left(  I\left(  \kappa\right)  \right)  :\kappa\in W\right\}
$ (this also \emph{defines} the collection of $\mathcal{T}$-tent regions).

At times we will identify a stopping time $W=W_{\mathcal{T}}$ in a tree
$\mathcal{T}$ with its associated $\mathcal{T}$-shadow on the circle and its
$\mathcal{T}$-tent region in the disk and will use $W\ $or $W_{\mathcal{T}}$
to denote any of them. When we do this the exact interpretation will be clear
from the context.

Note that for any open subset $E$ of the circle $\mathbb{T}$, there is a
unique $\mathcal{T}$-open set $G\subset E$ such that $E\setminus G$ is at most
countable. We often informally identify the open sets $E$ and $G$.

For a functions $k,$ $K$ defined on $\mathcal{T}$ set%
\[
Ik\left(  x\right)  =\sum_{y\in\left[  o,x\right]  }k\left(  y\right)  ,\text{
}\Delta K(x)=K(x)-K(x^{-})
\]
with the convention that $K(o^{-})=0.$

For $\Omega\subseteq\mathcal{T}$ a point $x\in\mathcal{T}$ is in the interior
of $\Omega$ if $x,x^{-1},$ $x_{+},$ $x_{-}\in\Omega$. A function $H$ is
\textit{harmonic} in $\Omega$ if
\begin{equation}
H(x)=\frac{1}{3}[H(x^{-1})+H(x_{+})+H(x_{-})]\label{mvp}%
\end{equation}
for every point $x$ which is interior in $\Omega$. If $H=Ih$ is harmonic then
for all $x$ in the interior of $\Omega$
\begin{equation}
h(x)=h(x_{+})+h(x_{-}).\label{addition}%
\end{equation}

Let $Cap_{\mathcal{T}}$ be the tree capacity associated with $\mathcal{T}$:%
\begin{equation}
Cap_{\mathcal{T}}\left(  E\right)  =\inf\left\{  \left\Vert f\right\Vert
_{\ell^{2}(\mathcal{T)}}^{2}:If\geq1\text{ on }E\right\}  .\label{captree}%
\end{equation}
More generally, if $E,F\subset$ $\mathcal{T}$ are disjoint stopping times with
$F\succ E,$ the capacity of the pair $\left(  E,F\right)  $, commonly known as
a condenser, is given by%
\begin{equation}
Cap_{\mathcal{T}}\left(  E,F\right)  =\inf\left\{  \left\Vert f\right\Vert
_{\ell^{2}(\mathcal{T)}}^{2}:If\geq1\text{ on }F,\text{ }\mbox{supp}(f)\subset%
{\textstyle\bigcup\limits_{e\in E}}
S(e)\right\}  .\label{capcondenser}%
\end{equation}

Let $\mathcal{T}_{\theta}$ be the rotation of the tree $\mathcal{T}$ by the
angle $\theta$, and let $Cap_{\mathcal{T}_{\theta}}$ be the tree capacity
associated with $\mathcal{T}_{\theta}$ as in (\ref{captree}), and extend the
definition to open subsets $G$ of the circle $\mathbb{T}$ by,%
\[
Cap_{\mathcal{T}_{\theta}}\left(  G\right)  =\inf\left\{  \sum_{\kappa
\in\mathcal{T}_{\theta}}f\left(  \kappa\right)  ^{2}:If\left(  \beta\right)
\geq1\text{ for }\beta\in\mathcal{T}_{\theta}\text{, }I\left(  \beta\right)
\subset G\right\}  .
\]
This is consistent with the definition of tree capacity of a stopping time $W
$ in $\mathcal{T}_{\theta};$ that is, if $G=\cup\left\{  I\left(
\kappa\right)  :\kappa\in W\right\}  $ we have
\[
Cap_{\mathcal{T}_{\theta}}\left(  W\right)  =Cap_{\mathcal{T}_{\theta}}\left(
\left\{  o\right\}  ,W\right)  =Cap_{\mathcal{T}_{\theta}}\left(  G\right)  .
\]
When the angle $\theta$ is not important, we will simply write $\mathcal{T}$
with the understanding that all results have analogues with $\mathcal{T}%
_{\theta}$ in place of $\mathcal{T}$.

We will use functions on the disk which are approximate extremals for
measuring capacity, that is functions for which the equality in
(\ref{extremal}) is approximately attained. A tool in doing that is an
analysis of the model problems on a tree. The following result about tree
capacities and extremals is proved in the Appendix.

\begin{proposition}
\label{extcondenser}Suppose $E,$ $F\subset$ $\mathcal{T}$ are disjoint
stopping times with $F\succ E.$

\begin{enumerate}
\item There is an extremal function $H=Ih$ such that $Cap(E,F)=\Vert
h\Vert_{\ell^{2}}^{2}$.

\item The function $H$ is harmonic on $\mathcal{T}\setminus(E\cup F)$.

\item If $S$ is a stopping time in $\mathcal{T}$, then $\sum_{\kappa\in S}$
$\left\vert h\left(  \kappa\right)  \right\vert \leq2Cap(E,F)$.

\item The function $h$ is positive on $\mathcal{G}\left(  E,F\right)  $, and
zero elsewhere.
\end{enumerate}
\end{proposition}

\begin{definition}
[stopping time blowup]\label{stopblowup}Given $0\leq\rho\leq1$ and a stopping
time $W$ in\ a tree $\mathcal{T}$, define the \emph{stopping time blowup}
$W_{\mathcal{T}}^{\rho}$ of $W$ in $\mathcal{T}$ as the set of minimal tree
elements in $\left\{  R^{\rho}\kappa:\kappa\in\mathcal{T}_{\theta}\right\}  $,
where $R^{\rho}\kappa$ denotes the unique element in the tree $\mathcal{T}$
satisfying%
\begin{align}
o  & \leq R^{\rho}\kappa\leq\kappa,\label{rhoroot}\\
\rho d\left(  \kappa\right)   & \leq d\left(  R^{\rho}\kappa\right)  <\rho
d\left(  \kappa\right)  +1.\nonumber
\end{align}

\end{definition}

Clearly $W_{\mathcal{T}}^{\rho}$ is a stopping time in $\mathcal{T}$. Note
that $R^{1}\kappa=\kappa$. The element $R^{\rho}\kappa$ can be thought of as
the "$\rho^{th}$ root of $\kappa$" since $\left\vert R^{\rho}\kappa\right\vert
=2^{-d\left(  R^{\rho}\kappa\right)  }\approx2^{-\rho d\left(  \kappa\right)
}=\left\vert \kappa\right\vert ^{\rho}$.

If $W$ is a stopping time for $\mathcal{T}$ and $W_{\mathcal{T}}^{\rho}$ is
the stopping time blowup of $W$, then there is a good estimate for the tree
capacity of $W_{\mathcal{T}}^{\rho}$ given in Lemma \ref{contain} below:
$Cap_{\mathcal{T}}\left(  \left\{  o\right\}  ,W_{\mathcal{T}}^{\rho}\right)
\leq\rho^{-2}Cap_{\mathcal{T}}\left(  \left\{  o\right\}  ,W\right)  .$
Unfortunately there is not a good condenser estimate of the form
$Cap_{\mathcal{T}}\left(  W_{\mathcal{T}}^{\rho},W\right)  \leq C_{\rho
}Cap_{\mathcal{T}}\left(  \left\{  o\right\}  ,W\right)  ;$ the left side can
be infinite when the right side is finite. We now introduce another type of
blowup, a tree analog of the disk blowup, for which we do have an effective
condenser estimate. We do this using a capacitary extremal function and a
comparison principle. Let $W$ be a stopping time in $\mathcal{T}$. By
Proposition \ref{extcondenser}, there is a unique extremal function $H=Ih$
such that
\begin{align}
Ih(x)  & =H\left(  x\right)  =1\text{ for }x\in W,\label{ext}\\
Cap_{\mathcal{T}}W  & =\left\Vert h\right\Vert _{\ell^{2}}^{2}.\nonumber
\end{align}

\begin{definition}
[capacitary blowup]\label{capblowup}Given a stopping time $W$ in $\mathcal{T}%
$, the corresponding extremal $H$ satisfying (\ref{ext}), and $0<\rho<1$,
define the \emph{capacitary blowup} $\widehat{W_{\mathcal{T}}^{\rho}}$ of $W$
by%
\[
\widehat{W_{\mathcal{T}}^{\rho}}=\left\{  t\in\mathcal{G}\left(  \left\{
o\right\}  ,W\right)  :H\left(  t\right)  \geq\rho\text{ and }H\left(
x\right)  \leq\rho\text{ for }x<t\right\}  .
\]

\end{definition}

Clearly $\widehat{W_{\mathcal{T}}^{\rho}}$ is a stopping time in $\mathcal{T}$.

\begin{lemma}
\label{newblowup}$Cap_{\mathcal{T}}\widehat{W_{\mathcal{T}}^{\rho}}\leq
\rho^{-2}Cap_{\mathcal{T}}W.$

\begin{proof}
Let $H$ be the extremal for $W$ in (\ref{ext}) and set $h=\Delta H,$ $h^{\rho
}=\frac{1}{\rho}h$ and $H^{\rho}=\frac{1}{\rho}H.$ Then $H^{\rho} $ is a
candidate for the infimum in the definition of capacity of $\widehat
{W_{\mathcal{T}}^{\rho}}$, and hence by the "comparison principle",
\[
Cap_{\mathcal{T}}\widehat{W_{\mathcal{T}}^{\rho}}\leq\left\Vert h^{\rho
}\right\Vert _{\ell^{2}}^{2}=\left(  \frac{1}{\rho}\right)  ^{2}\left\Vert
h\right\Vert _{\ell^{2}}^{2}=\rho^{-2}Cap_{\mathcal{T}}W.
\]

\end{proof}
\end{lemma}

The next lemma is used in the proof of our main estimate, (\ref{maintheta})
and it requires an upper bound on $Cap_{\mathbb{D}}\left(  G\right)  .$
However (\ref{maintheta}) is straightforward if $Cap_{\mathbb{D}}\left(
G\right)  $ bounded away from zero so that restriction is not a problem. In
fact, moving forward we will assume, at times implicitly, that
$Cap_{\mathbb{D}}\left(  G\right)  $ is not large.

\begin{lemma}
\label{newcondenser}$Cap_{\mathcal{T}}\left(  W,\widehat{W_{\mathcal{T}}%
^{\rho}}\right)  \leq\frac{4}{\left(  1-\rho\right)  ^{2}}Cap_{\mathcal{T}}W $
provided $Cap_{\mathcal{T}}W\leq\left(  1-\rho\right)  ^{2}/4$.

\begin{proof}
Let $H$ be the extremal for $W$ in (\ref{ext}). For $t\in\widehat
{W_{\mathcal{T}}^{\rho}}$ we have by our assumption,
\[
h\left(  t\right)  \leq\left\Vert h\right\Vert _{\ell^{2}}\leq\sqrt
{Cap_{\mathcal{T}}W}\leq\frac{1}{2}\left(  1-\rho\right)  ,
\]
and so%
\[
H\left(  t\right)  =H\left(  t^{-}\right)  +h\left(  t\right)  \leq\rho
+\frac{1}{2}\left(  1-\rho\right)  =\frac{1+\rho}{2}.
\]
If we define $\widetilde{H}\left(  t\right)  =\frac{2}{1-\rho}\left\{
H\left(  t\right)  -\frac{1+\rho}{2}\right\}  $, then $\widetilde{H}\leq0$ on
$\widehat{W_{\mathcal{T}}^{\rho}}$ and $\widetilde{H}=1$ on $W$. Thus
$\widetilde{H}$ is a candidate for the capacity of the condenser and so by the
"comparison principle"
\begin{align*}
Cap_{\mathcal{T}}\left(  W,\widehat{W_{\mathcal{T}}^{\rho}}\right)   &
\leq\left\Vert \bigtriangleup\widetilde{H}\right\Vert _{\ell^{2}\left(
\mathcal{G}\left(  W_{\mathcal{T}}^{\rho},W\right)  \right)  }^{2}%
\leq\left\Vert \bigtriangleup\widetilde{H}\right\Vert _{\ell^{2}\left(
\mathcal{T}\right)  }^{2}\\
& =\left(  \frac{2}{1-\rho}\right)  ^{2}\left\Vert h\right\Vert _{\ell
^{2}\left(  \mathcal{T}\right)  }^{2}=\frac{4}{\left(  1-\rho\right)  ^{2}%
}Cap_{\mathcal{T}}W.
\end{align*}

\end{proof}
\end{lemma}

We also have good \emph{tree} separation inherited from the stopping time
blowup $W_{\mathcal{T}}^{\rho}$. This gives our substitute for (\ref{B}) and
(\ref{lim}).

\begin{lemma}
\label{contain}$W_{\mathcal{T}}^{\rho}\subset\widehat{W_{\mathcal{T}}^{\rho}}$
as open subsets of the circle or, equivalently, as $\mathcal{T}$-tent regions
in the disk. Consequently $Cap_{\mathcal{T}}W_{\mathcal{T}}^{\rho}\leq
\rho^{-2}Cap_{\mathcal{T}}W$.

\begin{proof}
The restriction of $H$ to a geodesic is a concave function of distance from
the root, and so if $o<z<w\in W$, then%
\[
H\left(  z\right)  \geq\left(  1-\frac{d\left(  z\right)  }{d\left(  w\right)
}\right)  H\left(  o\right)  +\frac{d\left(  z\right)  }{d\left(  w\right)
}H\left(  w\right)  =\frac{d\left(  z\right)  }{d\left(  w\right)  }\geq
\rho,\ z\in\widehat{W_{\mathcal{T}}^{\rho}},
\]
and this proves $W_{\mathcal{T}}^{\rho}\subset\widehat{W_{\mathcal{T}}^{\rho}%
}$. The inequality now follows from Lemma \ref{newblowup}.
\end{proof}
\end{lemma}

\subsection{Holomorphic Approximate Extremals\label{approximate} and Capacity
Estimates}

We now define a holomorphic approximation $\Phi$ to the extremal function
$H=Ih$ on $\mathcal{T}$ constructed in Proposition \ref{extcondenser}. We will
use a parameter $s$. We always suppose $s>-1$ and additional specific
assumptions will be made at various places$.$ Define $\varphi_{\kappa}\left(
z\right)  =\left(  \frac{1-\left\vert \kappa\right\vert ^{2}}{1-\overline
{\kappa}z}\right)  ^{1+s}$ and set%
\begin{equation}
\Phi\left(  z\right)  =\sum_{\kappa\in\mathcal{T}}h\left(  \kappa\right)
\varphi_{\kappa}\left(  z\right)  =\sum_{\kappa\in\mathcal{T}}h\left(
\kappa\right)  \left(  \frac{1-\left\vert \kappa\right\vert ^{2}}%
{1-\overline{\kappa}z}\right)  ^{1+s}.\label{defFw}%
\end{equation}
Note that for $\tau\in\mathcal{T}$%
\[
\sum_{\kappa\in\mathcal{T}}h\left(  \kappa\right)  I\delta_{\kappa}\left(
\mathcal{\tau}\right)  =I\left(  \sum_{\kappa\in\mathcal{T}}h\left(
\kappa\right)  \delta_{\kappa}\right)  \left(  \mathcal{\tau}\right)
=Ih\left(  \mathcal{\tau}\right)  =H\left(  \mathcal{\tau}\right)  ,
\]
and so%
\begin{equation}
\Phi\left(  z\right)  -H\left(  z\right)  =\sum_{\kappa\in\mathcal{T}}h\left(
\kappa\right)  \left\{  \varphi_{\kappa}-I\delta_{\kappa}\right\}  \left(
z\right)  .\label{comparison}%
\end{equation}

Define $\Gamma_{s}$ by
\begin{equation}
\Gamma_{s}h\left(  z\right)  =\int_{\mathbb{D}}h\left(  \zeta\right)
\frac{(1-\left\vert \zeta\right\vert ^{2})^{s}}{\left(  1-\overline{\zeta
}z\right)  ^{1+s}}dA,\label{projop}%
\end{equation}
and recall that for appropriate constant $c_{s},$ $c_{s}\Gamma_{s}$ is a
projection onto holomorphic functions \cite[Thm 2.11]{Zhu}. For notational
convenience we absorb the constant $c_{s}$ into the measure $dA.$ Thus for
$h\in\mathcal{P}(\mathbb{D}),$%
\begin{equation}
\Gamma_{s}h\left(  z\right)  =h\left(  z\right)  .\label{repro}%
\end{equation}
We then have $\Phi=\Gamma_{s}g$ where%
\begin{equation}
g\left(  \zeta\right)  =\sum_{\kappa\in\mathcal{T}}h\left(  \kappa\right)
\frac{1}{\left\vert B_{\kappa}\right\vert }\frac{\left(  1-\overline{\zeta
}\kappa\right)  ^{1+s}}{(1-\left\vert \zeta\right\vert ^{2})^{s}}%
\chi_{B_{\kappa}}\left(  \zeta\right)  ,\label{defgw}%
\end{equation}
and $B_{\kappa}$ is the Euclidean ball centered at $\kappa$ with radius
$c\left(  1-\left\vert \kappa\right\vert \right)  $ where $c$ is a small
positive constant to be chosen later. The function $\Phi$ satisfies the
following estimates.

\begin{proposition}
\label{newBoe}Set $F=\widehat{E_{\mathcal{T}}^{\rho}}$ and write $E=\left\{
w_{k}\right\}  _{k}$. Suppose $z\in\mathbb{D}$ and $s>-1$. Then we have%
\begin{equation}
\left\{
\begin{array}
[c]{llll}%
\left\vert \Phi\left(  z\right)  -\Phi\left(  w_{k}\right)  \right\vert  &
\leq & CCap_{\mathcal{T}}\left(  E,F\right)  , & z\in T\left(  w_{k}\right) \\
\text{ }\operatorname{Re}\Phi\left(  w_{k}\right)  & \geq & c>0, & k\geq1\\
\text{ }\left\vert \Phi\left(  w_{k}\right)  \right\vert  & \leq & C, &
k\geq1\\
\text{ }\left\vert \Phi\left(  z\right)  \right\vert  & \leq &
CCap_{\mathcal{T}}\left(  E,F\right)  , & z\notin F.
\end{array}
\right.  .\label{estimates}%
\end{equation}

\end{proposition}

\begin{corollary}
\label{smallcap}Furthermore, if $s>-\frac{1}{2}$ then $\Phi=\Gamma_{s}g$ for a
$g$ which satisfies
\begin{equation}
\int_{\mathbb{D}}\left\vert g\left(  \zeta\right)  \right\vert ^{2}dA\leq
C\ Cap_{\mathcal{T}}\left(  E,F\right)  ;\label{maincon}%
\end{equation}
and if $s>\frac{1}{2}$ then
\begin{equation}
\left\Vert \Phi\right\Vert _{\mathcal{D}}^{2}\leq\int_{\mathbb{D}}\left\vert
g\left(  \zeta\right)  \right\vert ^{2}dA\leq C\ Cap_{\mathcal{T}}\left(
E,F\right)  .\label{Phicapacity}%
\end{equation}

\begin{proof}
From (\ref{comparison}) we have%
\begin{align*}
\left\vert \Phi\left(  z\right)  -H\left(  z\right)  \right\vert  & \leq
\sum_{\kappa\in\left[  o,z\right]  }\left\vert h\left(  \kappa\right)
\left\{  \varphi_{\kappa}\left(  z\right)  -1\right\}  \right\vert
+\sum_{\kappa\notin\left[  o,z\right]  }\left\vert h\left(  \kappa\right)
\varphi_{\kappa}\left(  z\right)  \right\vert \\
& =I\left(  z\right)  +II\left(  z\right)  .
\end{align*}
We also have that $h$ is nonnegative and supported in $V_{G}^{\gamma}\setminus
V_{G}^{\alpha}$. We first show that%
\[
II\left(  z\right)  \leq\sum_{\kappa\notin\left[  o,z\right]  }h\left(
\kappa\right)  \left\vert \frac{1-\left\vert \kappa\right\vert ^{2}%
}{1-\overline{\kappa}z}\right\vert ^{1+s}\leq CCap\left(  E,F\right)  .
\]
For $A>1$ let%
\[
\Omega_{j}=\left\{  \kappa\in\mathcal{T}:A^{-j-1}<\left\vert \frac
{1-\left\vert \kappa\right\vert ^{2}}{1-\overline{\kappa}z}\right\vert \leq
A^{-j}\right\}  .
\]

\begin{lemma}
For every $j$ the set $\Omega_{j}$ is a union of two stopping times for
$\mathcal{T}$.

\begin{proof}
Let $\Omega_{j}^{\text{1}}\ $be the subset of $\Omega_{j}$ of points whose
distance from the root is odd and set $\Omega_{j}^{\text{2}}=\Omega
_{j}\setminus\Omega_{j}^{\text{1}}.$ We will show both are stopping times;
i.e. if for $r=1,2,$ $\kappa\in\Omega_{j}^{r}$, $\lambda\in\mathcal{T}$, and
$\kappa\in\lbrack o,\lambda)$, then $\lambda\notin\Omega_{j}^{r}. $

Set $\delta\kappa=\lambda-\kappa$. We have
\begin{align}
\left\vert \frac{1-\bar{\lambda}z}{1-\left\vert \lambda\right\vert ^{2}%
}\right\vert  & =\frac{1-\left\vert \kappa\right\vert ^{2}}{1-\left\vert
\lambda\right\vert ^{2}}\left\vert \frac{1-\left(  \overline{\kappa
+\delta\kappa}\right)  z}{1-\left\vert \kappa\right\vert ^{2}}\right\vert
\nonumber\\
& =\frac{1-\left\vert \kappa\right\vert ^{2}}{1-\left\vert \lambda\right\vert
^{2}}\left\vert \frac{1-\bar{\kappa}z}{1-\left\vert \kappa\right\vert ^{2}%
}-\frac{\overline{\delta\kappa}z}{1-\left\vert \kappa\right\vert ^{2}%
}\right\vert \nonumber\\
& \geq\frac{1-\left\vert \kappa\right\vert ^{2}}{1-\left\vert \lambda
\right\vert ^{2}}\left\{  \left\vert \frac{1-\bar{\kappa}z}{1-\left\vert
\kappa\right\vert ^{2}}\right\vert -\frac{\left\vert \overline{\delta\kappa
}z\right\vert }{1-\left\vert \kappa\right\vert ^{2}}\right\} \label{*}%
\end{align}
By the construction of the tree $(1-\left\vert \kappa\right\vert ^{2}%
)\sim2^{s}(1-\left\vert \lambda\right\vert ^{2})$ for some positive integer
$s,$ and if $\kappa$ and $\lambda$ are in the same $\Omega_{j}^{r}$ then
$s\geq2.$ Also, by the construction of $\mathcal{T},$ we have
\[
\frac{\left\vert \overline{\delta\kappa}z\right\vert }{1-\left\vert
\kappa\right\vert ^{2}}\leq\frac{\sqrt{2}\left(  1-\left\vert \kappa
\right\vert \right)  \left\vert z\right\vert }{1-\left\vert \kappa\right\vert
^{2}}\lesssim\frac{\sqrt{2}}{2},
\]
and hence we continue with
\[
\left\vert \frac{1-\bar{\lambda}z}{1-\left\vert \lambda\right\vert ^{2}%
}\right\vert \geq4\left(  A^{j}-\frac{\sqrt{2}}{2}\right)  .
\]
We are done if for each $j,$ $A^{j+1}\leq4\left(  A^{j}-\sqrt{2}/2\right)  $.
That holds if $A\leq4(1-\sqrt{2}/2)<1.\,17.$
\end{proof}
\end{lemma}

Now by the stopping time property, item 3 in Proposition \ref{extcondenser},
we have%
\[
\sum_{\kappa\in\Omega_{j}}h\left(  \kappa\right)  \leq CCap_{\mathcal{T}%
}\left(  E,F\right)  ,\ j\geq0.
\]
Altogether we then have%
\[
II\left(  z\right)  \leq\sum_{j=0}^{\infty}\sum_{\kappa\in\Omega_{j}}h\left(
\kappa\right)  A^{-j\left(  1+s\right)  }\leq C_{s}Cap_{\mathcal{T}}\left(
E,F\right)  .
\]

If $z\in\mathbb{D}\setminus F$ then $I\left(  z\right)  =0$ and $H\left(
z\right)  =0$ and we have
\[
\left\vert \Phi\left(  z\right)  \right\vert =\left\vert \Phi\left(  z\right)
-H\left(  z\right)  \right\vert \leq II\left(  z\right)  \leq C_{s}%
Cap_{\mathcal{T}}\left(  E,F\right)  ,
\]
which is the fourth line in (\ref{estimates}).

If $z\in T\left(  w_{j}\right)  $, then for $\kappa\notin\left[
o,w_{j}\right]  $ we have
\[
\left\vert \varphi_{\kappa}\left(  w_{j}\right)  \right\vert \leq C\left\vert
\varphi_{\kappa}\left(  z\right)  \right\vert ,
\]
and for $\kappa\in\left[  o,z\right]  $ we have%
\[
\left\vert \varphi_{\kappa}\left(  z\right)  -\varphi_{\kappa}\left(
w_{j}\right)  \right\vert =\left\vert \left(  \frac{1-\left\vert
\kappa\right\vert ^{2}}{1-\overline{\kappa}z}\right)  ^{1+s}-\left(
\frac{1-\left\vert \kappa\right\vert ^{2}}{1-\overline{\kappa}w_{j}}\right)
^{1+s}\right\vert \leq C_{s}\frac{\left\vert z-w_{j}\right\vert }{1-\left\vert
\kappa\right\vert ^{2}}.
\]
Thus for $z\in T\left(  w_{j}^{\alpha}\right)  $,%
\begin{align*}
\left\vert \Phi\left(  z\right)  -\Phi\left(  w_{j}\right)  \right\vert  &
\leq\sum_{\kappa\in\left[  o,w_{j}^{\alpha}\right]  }h\left(  \kappa\right)
\left\vert \varphi_{\kappa}\left(  z\right)  -\varphi_{\kappa}\left(
w_{j}\right)  \right\vert +C\sum_{\kappa\notin\left[  o,z\right]  }h\left(
\kappa\right)  \left\vert \varphi_{\kappa}\left(  z\right)  \right\vert \\
& \leq C_{s}\sum_{\kappa\in\left[  o,w_{j}^{\alpha}\right]  }h\left(
\kappa\right)  \frac{\left\vert z-w_{j}\right\vert }{1-\left\vert
\kappa\right\vert ^{2}}+CII\left(  z\right) \\
& \leq C_{s}Cap_{\mathcal{T}}\left(  E,F\right)  ,
\end{align*}
since $h\left(  \kappa\right)  \leq C\ Cap_{\mathcal{T}}\left(  E,F\right)  $
and $\sum_{\kappa\in\left[  o,w_{j}\right]  }\frac{1}{1-\left\vert
\kappa\right\vert ^{2}}\approx\frac{1}{1-\left\vert w_{j}\right\vert ^{2}}$.
This proves the first line in (\ref{estimates}).

Moreover, we note that for $s=0$ and $\kappa\in\left[  o,w_{j}\right]  $,%
\[
\operatorname{Re}\varphi_{\kappa}\left(  w_{j}\right)  =\operatorname{Re}%
\frac{1-\left\vert \kappa\right\vert ^{2}}{1-\overline{\kappa}w_{j}%
}=\operatorname{Re}\frac{1-\left\vert \kappa\right\vert ^{2}}{\left\vert
1-\overline{\kappa}w_{j}\right\vert ^{2}}\left(  1-\kappa\overline{w_{j}%
}\right)  \geq c>0.
\]
A similar result holds for $s>-1$ provided the Bergman tree $\mathcal{T}$ is
constructed sufficiently thin depending on $s$. It then follows from
$\sum_{\kappa\in\left[  o,w_{j}\right]  }h\left(  \kappa\right)  =1$ that
\begin{align*}
\operatorname{Re}\Phi\left(  w_{j}\right)   & =\sum_{\kappa\in\left[
o,w_{j}\right]  }h\left(  \kappa\right)  \operatorname{Re}\varphi_{\kappa
}\left(  w_{j}\right)  +\sum_{\kappa\notin\left[  o,w_{j}\right]  }h\left(
\kappa\right)  \operatorname{Re}\varphi_{\kappa}\left(  w_{j}\right) \\
& \geq c\sum_{\kappa\in\left[  o,w_{j}\right]  }h\left(  \kappa\right)
-C\ Cap_{\mathcal{T}}\left(  E,F\right)  \geq c^{\prime}>0.
\end{align*}
We trivially have%
\[
\left\vert \Phi\left(  w_{j}\right)  \right\vert \leq I\left(  z\right)
+II\left(  z\right)  \leq C\sum_{\kappa\in\left[  o,w_{j}\right]  }h\left(
\kappa\right)  +C\ Cap_{\mathcal{T}}\left(  E,F\right)  \leq C,
\]
and this completes the proof of (\ref{estimates}).

Now we prove (\ref{maincon}). From property 1 of Proposition
\ref{extcondenser} we obtain%
\begin{align*}
\int_{\mathbb{D}}\left\vert g\left(  \zeta\right)  \right\vert ^{2}dA  &
=\int_{\mathbb{D}}\left\vert \sum_{\kappa\in\mathcal{T}}h\left(
\kappa\right)  \frac{1}{\left\vert B_{\kappa}\right\vert }\frac{\left(
1-\overline{\zeta}\kappa\right)  ^{1+s}}{(1-\left\vert \zeta\right\vert
^{2})^{s}}\chi_{B_{\kappa}}\left(  \zeta\right)  \right\vert ^{2}dA\\
& =\sum_{\kappa\in\mathcal{T}}\left\vert h\left(  \kappa\right)  \right\vert
^{2}\frac{1}{\left\vert B_{\kappa}\right\vert ^{2}}\int_{B_{\kappa}}%
\frac{\left\vert 1-\overline{\zeta}\kappa\right\vert ^{2+2s}}{(1-\left\vert
\zeta\right\vert ^{2})^{2s}}dA\\
& \approx\sum_{\kappa\in\mathcal{T}}\left\vert h\left(  \kappa\right)
\right\vert ^{2}\approx Cap_{\mathcal{T}}\left(  E,F\right)  .
\end{align*}
Finally (\ref{Phicapacity}) follows from (\ref{maincon}) and Lemma 2.4 of
\cite{Boe}.
\end{proof}
\end{corollary}

\begin{corollary}
\label{capequiv}Let $G$ be a finite union of arcs in the circle $\mathbb{T}$.
Then%
\begin{equation}
Cap_{\mathbb{D}}\left(  G\right)  \approx Cap_{\mathcal{T}}\left(  G\right)
,\label{claimSteg}%
\end{equation}
where $Cap_{\mathbb{D}}$ denotes Stegenga's capacity on the circle
$\mathbb{T}$.

\begin{proof}
To prove the inequality $\lessapprox$ in (\ref{claimSteg}) we use Proposition
\ref{newBoe} to obtain a test function for estimating the Stegenga capacity of
$G$. We take $F=\left\{  o\right\}  $ and $E=G$ in Proposition \ref{newBoe}.
Let $c,C$ be the constants in Proposition \ref{newBoe}, and suppose that
$Cap\left(  E,F\right)  \leq c/(3C)$. Set $\Psi\left(  z\right)  =\frac{3}%
{c}\left(  \Phi\left(  z\right)  -\Phi\left(  0\right)  \right)  $. Then
$\Psi\left(  0\right)  =0$ and
\begin{align*}
\operatorname{Re}\Psi\left(  z\right)   & =\frac{3}{c}\left\{
\operatorname{Re}\Phi\left(  z\right)  -\operatorname{Re}\Phi\left(  0\right)
\right\} \\
& \geq\frac{3}{c}\left\{  c-2CCap_{\mathcal{T}}\left(  E,F\right)  \right\}
\geq1,\ z\in G.
\end{align*}
By definition (\ref{extremal}) and (\ref{Phicapacity}) we have that for
$G\subset\mathbb{T}$%
\begin{align*}
Cap_{\mathbb{D}}(G)  & \leq\left\Vert \Psi\right\Vert _{\mathcal{D}}%
^{2}=\left(  \frac{3}{c}\right)  ^{2}\left\Vert \Phi\right\Vert _{\mathcal{D}%
}^{2}\\
& \leq\left(  \frac{3}{c}\right)  ^{2}C\ Cap_{\mathcal{T}}\left(  E,F\right)
\leq C\ Cap_{\mathcal{T}}E\\
& =C\ Cap_{\mathcal{T}}G.
\end{align*}

To obtain the opposite inequality we use $\psi\in\mathcal{D},$ an extremal
function for computing $Cap_{\mathbb{D}}G$. For $R>0,$ $z\in\mathbb{D}$ let
$B(z,R)$ be the hyperbolic disk of radius $R$ centered at $z$. Pick $R$ large
enough so that for all $\kappa\in\mathcal{T}\setminus\left\{  o\right\}  $ we
have $B(\kappa,R)\supset\operatorname{convexhull}\left(  B_{\kappa}\cup
B_{\kappa^{-1}}\right)  .$ Our candidate for estimating $Cap_{\mathcal{T}}$ is
given by setting $h\left(  o\right)  =0$ and
\[
h\left(  \kappa\right)  =(1-\left\vert \kappa\right\vert ^{2})\sup\left\{
\left\vert \psi^{\prime}(z)\right\vert :z\in B(\kappa,R)\right\}  ;\text{
}\kappa\in\mathcal{T}\setminus\left\{  o\right\}  .
\]
We have the pointwise estimate
\begin{align*}
\operatorname{Re}\psi\left(  \beta\right)   & \leq\left\vert \psi\left(
\beta\right)  \right\vert \leq\sum_{\kappa\in\left[  o,\beta\right]
}\left\vert \psi\left(  \kappa\right)  -\psi\left(  \kappa^{-1}\right)
\right\vert \\
& \leq\sum_{\kappa\in\left[  o,\beta\right]  }\left\vert \kappa-\kappa
^{-1}\right\vert \sup\left\{  \left\vert \psi^{\prime}(z)\right\vert
:z\in\operatorname{segment}\left(  \kappa,\kappa^{-1}\right)  \right\} \\
& \leq C\sum_{\kappa\in\left[  o,\beta\right]  }h(\kappa)=CIh(\beta).
\end{align*}
We have the norm estimate, with $z\left(  \kappa\right)  $ denoting the
appropriate point in $B(\kappa,R),$%
\begin{align*}
\left\Vert h\right\Vert _{\ell^{2}\left(  \mathcal{T}\right)  }^{2}  &
=\sum_{\kappa\in\mathcal{T}}\left(  1-\left\vert \kappa^{2}\right\vert
\right)  ^{2}\left\vert \psi^{\prime}\left(  z\left(  \kappa\right)  \right)
\right\vert ^{2}\\
& \leq C\sum_{\kappa\in\mathcal{T}}\frac{\left(  1-\left\vert \kappa
^{2}\right\vert \right)  ^{2}}{\left\vert B(\kappa,R)\right\vert }%
\int_{B(\kappa,R)}\left\vert \psi^{\prime}\left(  z\right)  \right\vert
^{2}dA\\
& \leq C\sum_{\kappa\in\mathcal{T}}\int_{B(\kappa,R)}\left\vert \psi^{\prime
}\left(  z\right)  \right\vert ^{2}dA\\
& \leq C\int_{\mathbb{D}}\left\vert \psi^{\prime}\left(  z\right)  \right\vert
^{2}dA\leq C\left\Vert \psi\right\Vert _{\mathcal{D}}^{2}.
\end{align*}
Here the first inequality uses he submean value property for the subharmonic
function $\left\vert \psi^{\prime}\left(  z\right)  \right\vert ^{2},$ the
second uses straightforward estimates for $\left\vert B(\kappa,R)\right\vert
,$ and the next estimate holds because the $B(\kappa,R)$ are approximately
disjoint; $\sum\chi_{B(\kappa,R)}(z)\leq C.$ Recalling definition
(\ref{captree}) we find
\[
Cap_{\mathcal{T}}G\leq C\left\Vert \frac{1}{c}\psi\right\Vert _{\mathcal{D}%
}^{2}=\frac{C}{c^{2}}Cap_{\mathbb{D}}G.
\]

\end{proof}
\end{corollary}

Abbreviate $Cap_{\mathcal{T}_{\theta}}$ by $Cap_{\theta}$, and let $T_{\theta
}\left(  E\right)  $ be the $\mathcal{T}_{\theta}$-tent region corresponding
to an open subset $E$ of the circle $\mathbb{T}$. Recall that $T\left(
E\right)  =%
{\textstyle\bigcup\limits_{I\subset E}}
T\left(  I\right)  $. Now define $M$ by%
\begin{equation}
M:=\sup_{E\text{ open }\subset\mathbb{T}}\frac{\int_{\mathbb{T}}\mu_{b}\left(
T_{\theta}\left(  E\right)  \right)  d\theta}{\int_{\mathbb{T}}Cap_{\theta
}\left(  E\right)  d\theta}.\label{Gmax}%
\end{equation}

\begin{corollary}
\label{M}We have $\left\Vert \mu_{b}\right\Vert _{\mathcal{D}-Carleson}%
^{2}\approx M.$

\begin{proof}
Using Corollary \ref{capequiv} and $T_{\theta}\left(  E\right)  \subset
T\left(  E\right)  $, we have%
\begin{align*}
M  & \leq C\sup_{E\text{ open }\subset\mathbb{T}}\frac{\int_{\mathbb{T}}%
\mu_{b}\left(  T\left(  E\right)  \right)  d\theta}{\int_{\mathbb{T}%
}Cap_{\mathbb{D}}\left(  E\right)  d\theta}\\
& =C\sup_{E\text{ open }\subset\mathbb{T}}\frac{\mu_{b}\left(  T\left(
E\right)  \right)  }{Cap_{\mathbb{D}}\left(  E\right)  }\approx\left\Vert
\mu_{b}\right\Vert _{\mathcal{D}-Carleson}^{2},
\end{align*}
where the final comparison is Stegenga's theorem. Conversely, one can verify
using an argument in the style of the one in (\ref{muest}) below that for
$0<\rho<1$,%
\begin{align*}
\mu_{b}\left(  E\right)   & \leq C\int_{\mathbb{T}}\mu_{b}\left(  T_{\theta
}\left(  E_{\mathbb{D}}^{\rho}\right)  \right)  d\theta\\
& \leq CM\int_{\mathbb{T}}Cap_{\theta}\left(  E_{\mathbb{D}}^{\rho}\right)
d\theta\\
& \approx CMCap_{\mathbb{D}}\left(  E_{\mathbb{D}}^{\rho}\right) \\
& \leq CMCap_{\mathbb{D}}\left(  E\right)  .
\end{align*}
Here the third line uses (\ref{claimSteg}) with $E_{\mathbb{D}}^{\rho}$ and
$\mathcal{T}\left(  \theta\right)  $ in place of $G$ and $\mathcal{T}$, and
the final inequality follows from\ (\ref{B}). Thus from Stegenga's theorem we
obtain%
\[
\left\Vert \mu_{b}\right\Vert _{\mathcal{D}-Carleson}^{2}\approx\sup_{E\text{
open }\subset\mathbb{T}}\frac{\mu_{b}\left(  E\right)  }{Cap_{\mathbb{D}%
}\left(  E\right)  }\leq CM.
\]

\end{proof}
\end{corollary}

Given $0<\delta<1$, let $G$ be an open set in $\mathbb{T}$ such that
\begin{equation}
\frac{\int_{\mathbb{T}}\mu_{b}\left(  T_{\theta}\left(  G\right)  \right)
d\theta}{\int_{\mathbb{T}}Cap_{\theta}\left(  G\right)  d\theta}\geq\delta
M.\label{defGdelta}%
\end{equation}
We need to know that $\mu_{b}(V_{G}^{\beta}\setminus V_{G})$ is small compared
to $\mu_{b}\left(  V_{G}\right)  $. This crucial step of the proof is where we
use the asymptotic capacity estimate Lemma \ref{contain}.

\begin{proposition}
\label{3}Given $\varepsilon>0$ we can choose $\delta=\delta(\varepsilon)<1$ in
(\ref{defGdelta}) and $\beta=\beta(\varepsilon)<1$ so that for any $G$
satisfying (\ref{defGdelta}), we have, with $V_{G}^{\beta}=G_{\mathbb{D}%
}^{\beta}$ and $V_{G}=G_{\mathbb{D}}^{1}=T\left(  G\right)  $,
\begin{equation}
\mu_{b}(V_{G}^{\beta}\setminus V_{G})\leq\varepsilon\mu_{b}\left(
V_{G}\right)  ,\label{inpart}%
\end{equation}

\begin{proof}
Let $G^{\rho}\left(  \theta\right)  =G_{\mathcal{T}_{\theta}}^{\rho}$. Lemma
\ref{contain} shows that $Cap_{\theta}\left(  G^{\rho}\left(  \theta\right)
\right)  \leq\rho^{-2}Cap_{\theta}\left(  G\right)  $,\ for $0\leq\theta<2\pi
$, $0<\rho<1$, and if we integrate on $\mathbb{T}$ we obtain%
\[
\int_{\mathbb{T}}Cap_{\theta}\left(  G^{\rho}\left(  \theta\right)  \right)
d\mathcal{\theta}\leq\rho^{-2}\int_{\mathbb{T}}Cap_{\theta}\left(  G\right)
d\theta.
\]
From (\ref{Gmax}) and (\ref{defGdelta}) we thus have%
\begin{align*}
\int_{\mathbb{T}}\mu_{b}\left(  T_{\theta}\left(  G^{\rho}\left(
\theta\right)  \right)  \right)  d\mathcal{\theta}  & \leq M\int_{\mathbb{T}%
}Cap_{\theta}\left(  G^{\rho}\left(  \theta\right)  \right)  d\theta\\
& \leq M\rho^{-2}\int_{\mathbb{T}}Cap_{\theta}\left(  G\right)  d\theta\\
& \leq\frac{1}{\delta\rho^{2}}\int_{\mathbb{T}}\mu_{b}\left(  T_{\theta
}\left(  G\right)  \right)  d\theta.
\end{align*}
It follows that%
\begin{align*}
\int\limits_{\mathbb{T}}\mu_{b}\left(  T_{\theta}\left(  G^{\rho}\left(
\theta\right)  \right)  \setminus T_{\theta}\left(  G\right)  \right)  d\theta
& =\int\limits_{\mathbb{T}}\mu_{b}\left(  T_{\theta}\left(  G^{\rho}\left(
\theta\right)  \right)  \right)  d\mathcal{\theta}-\int\limits_{\mathbb{T}}%
\mu_{b}\left(  T_{\theta}\left(  G\right)  \right)  d\theta\\
& \leq\left(  \frac{1}{\delta\rho^{2}}-1\right)  \int\limits_{\mathbb{T}}%
\mu_{b}\left(  T_{\theta}\left(  G\right)  \right)  d\theta.
\end{align*}

Now with $\eta=(\rho+1)/2$,%
\begin{align}
& \int\limits_{\mathbb{T}}\mu_{b}\left(  T_{\theta}\left(  G^{\rho}\left(
\theta\right)  \right)  \setminus T_{\theta}\left(  G\right)  \right)
d\theta=\int\limits_{\mathbb{T}}\int\limits_{T_{\theta}\left(  G^{\rho}\left(
\theta\right)  \right)  \setminus T_{\theta}\left(  G\right)  }d\mu_{b}\left(
z\right)  d\theta\label{muest}\\
& \qquad\qquad\qquad\geq\int_{\mathbb{T}}\int\limits_{T_{\theta}\left(
G^{\rho}\left(  \theta\right)  \right)  \setminus T\left(  G\right)  }d\mu
_{b}\left(  z\right)  d\theta\nonumber\\
& \qquad\qquad\qquad\geq\int_{\mathbb{T}}\int\limits_{T_{\theta}\left(
G^{\rho}\left(  \theta\right)  \right)  \setminus T\left(  G\right)  }d\mu
_{b}\left(  z\right)  d\theta\nonumber
\end{align}%
\[
=\int\limits_{\mathbb{D}}\left\{  \frac{1}{2\pi}\int\limits_{\left\{
\theta:z\in T_{\theta}\left(  G^{\rho}\left(  \theta\right)  \right)
\setminus T\left(  G\right)  \right\}  }d\mathcal{\theta}\right\}  d\mu
_{b}\left(  z\right)  \geq\frac{1}{2}\int_{T\left(  G_{\mathbb{D}}^{\eta
}\right)  \setminus T\left(  G\right)  }d\mu_{b}\left(  z\right)  ,
\]
since every $z\in T\left(  G_{\mathbb{D}}^{\eta}\right)  $ lies in $T_{\theta
}\left(  G^{\rho}\left(  \theta\right)  \right)  $ for at least half of the
$\theta$'s in $\left[  0,2\pi\right)  $. Here we may assume that the
components of $G_{\mathbb{D}}^{\rho}$ have small length since otherwise we
trivially have $\int_{\mathbb{T}}Cap_{\mathcal{T}\left(  \theta\right)
}\left(  G\right)  d\mathcal{\theta}\geq c>0.$ We continue with%
\begin{equation}
M\leq\frac{1}{c}\int d\mu_{b}\leq\frac{1}{c}\left\Vert b\right\Vert
_{\mathcal{D}}^{2}\leq\frac{C}{c}\left\Vert T_{b}\right\Vert ^{2}%
.\label{otherwise}%
\end{equation}

Combining the above inequalities, using $\rho=2\eta-1$, $1/2\leq\rho<1$, and
choosing $\delta=\eta$, we obtain%
\begin{align*}
\mu_{b}\left(  T\left(  G_{\mathbb{D}}^{\eta}\right)  \setminus T\left(
G\right)  \right)   & \leq2\left(  \frac{1}{\delta\rho^{2}}-1\right)
\int_{\mathbb{T}}\mu_{b}\left(  T_{\theta}\left(  G\right)  \right)  d\theta\\
& =2\left(  \frac{1}{\eta\left(  2\eta-1\right)  ^{2}}-1\right)
\int_{\mathbb{T}}\mu_{b}\left(  T_{\theta}\left(  G\right)  \right)  d\theta\\
& \leq C\left(  1-\eta\right)  \int_{\mathbb{T}}\mu_{b}\left(  T_{\theta
}\left(  G\right)  \right)  d\theta,
\end{align*}
for $3/4\leq\eta<1$. Recalling that $V_{G}^{\eta}=T\left(  G_{\mathbb{D}%
}^{\eta}\right)  $ and that for all $\theta$ we have $T_{\theta}\left(
G\right)  \subset T\left(  G\right)  =V_{G}$ this becomes
\[
\mu_{b}\left(  V_{G}^{\eta}\setminus V_{G}\right)  \leq C\left(
1-\eta\right)  \int_{\mathbb{T}}\mu_{b}\left(  T_{\theta}\left(  G\right)
\right)  d\theta\leq C\left(  1-\eta\right)  \mu_{b}\left(  V_{G}\right)
,\ \frac{3}{4}\leq\eta<1,
\]
Hence given $\varepsilon>0$ it is possible to select $\delta$ and $\beta$\ so
that (\ref{inpart}) holds.
\end{proof}
\end{proposition}

\subsection{Schur Estimates and a Bilinear Operator on Trees}

We begin with a bilinear version of Schur's well known theorem.

\begin{proposition}
\label{Schurtest}Let $\left(  X,\mu\right)  $, $\left(  Y,\nu\right)  $ and
$\left(  Z,\omega\right)  $ be measure spaces and $H\left(  x,y,z\right)  $ be
a nonnegative measurable function on $X\times Y\times Z$. Define, initially
for nonnegative functions $f,g,$%
\[
T\left(  f,g\right)  \left(  x\right)  =\int_{Y\times Z}H\left(  x,y,z\right)
f\left(  y\right)  d\nu\left(  y\right)  g\left(  z\right)  d\omega\left(
z\right)  ,\ x\in X,
\]
For $1<p<\infty$, suppose there are positive functions $h$, $k$ and $m$ on
$X$, $Y$ and $Z$ respectively such that%
\[
\int_{Y\times Z}H\left(  x,y,z\right)  k\left(  y\right)  ^{p^{\prime}%
}m\left(  z\right)  ^{p^{\prime}}d\nu\left(  y\right)  d\omega\left(
z\right)  \leq\left(  Ah\left(  x\right)  \right)  ^{p^{\prime}},
\]
for $\mu$-a.e.$\ x\in X$, and
\[
\int_{X}H\left(  x,y,z\right)  h\left(  x\right)  ^{p}d\mu\left(  x\right)
\leq\left(  Bk\left(  y\right)  m\left(  z\right)  \right)  ^{p},
\]
for $\nu\times\omega$-a.e.$\ \left(  y,z\right)  \in Y\times Z$. Then $T$ is
bounded from $L^{p}\left(  \nu\right)  \times L^{p}\left(  \omega\right)  $ to
$L^{p}\left(  \mu\right)  $ and $\left\Vert T\right\Vert _{operator}\leq AB.$
\end{proposition}

\begin{proof}
We have%
\begin{align*}
& \int_{X}\left\vert Tf\left(  x\right)  \right\vert ^{p}d\mu\left(  x\right)
\\
& \leq\int_{X}\left(  \int_{Y\times Z}H\left(  x,y,z\right)  k\left(
y\right)  ^{p^{\prime}}m\left(  z\right)  ^{p^{\prime}}d\nu\left(  y\right)
d\omega\left(  z\right)  \right)  ^{p/p^{\prime}}\\
& \times\left(  \int_{Y\times Z}H\left(  x,y,z\right)  \left(  \frac{f\left(
y\right)  }{k\left(  y\right)  }\right)  ^{p}d\nu\left(  y\right)  \left(
\frac{g\left(  z\right)  }{m\left(  z\right)  }\right)  ^{p}d\omega\left(
z\right)  \right)  d\mu\left(  x\right) \\
& \leq A^{p}\int_{Y\times Z}\left(  \int_{X}H\left(  x,y,z\right)  h\left(
x\right)  ^{p}d\mu\left(  x\right)  \right)  \left(  \frac{f\left(  y\right)
}{k\left(  y\right)  }\right)  ^{p}d\nu\left(  y\right)  \left(
\frac{g\left(  z\right)  }{m\left(  z\right)  }\right)  ^{p}d\omega\left(
z\right) \\
& \leq A^{p}B^{p}\int_{Y\times Z}k\left(  y\right)  ^{p}m\left(  z\right)
^{p}\left(  \frac{f\left(  y\right)  }{k\left(  y\right)  }\right)  ^{p}%
d\nu\left(  y\right)  \left(  \frac{g\left(  z\right)  }{m\left(  z\right)
}\right)  ^{p}d\omega\left(  z\right) \\
& =\left(  AB\right)  ^{p}\int_{Y}f\left(  y\right)  ^{p}d\nu\left(  y\right)
\int_{Z}g\left(  z\right)  ^{p}d\omega\left(  z\right)  .
\end{align*}

\end{proof}

This proposition can be used along with the estimates
\begin{equation}
\int_{\mathbb{D}}\frac{(1-\left\vert w\right\vert ^{2})^{t}}{\left\vert
1-\overline{w}z\right\vert ^{2+t+c}}dw\approx\left\{
\begin{array}
[c]{ccc}%
C_{t} & \text{ if } & c<0,\text{ }t>-1\\
-C_{t}\log(1-\left\vert z\right\vert ^{2}) & \text{ if } & c=0,\text{ }t>-1\\
C_{t}(1-\left\vert z\right\vert ^{2})^{-c} & \text{ if } & c>0,\text{ }t>-1
\end{array}
\right.  ,.\label{intest}%
\end{equation}
to prove a corollary we will use later \cite[Thm 2.10]{Zhu}.

\begin{corollary}
\label{210}Define%
\begin{align*}
Tf\left(  z\right)   & =(1-\left\vert z\right\vert ^{2})^{a}\int_{\mathbb{D}%
}\frac{(1-\left\vert w\right\vert ^{2})^{b}}{\left(  1-\overline{w}z\right)
^{2+a+b}}f\left(  w\right)  dw,\\
Sf\left(  z\right)   & =(1-\left\vert z\right\vert ^{2})^{a}\int_{\mathbb{D}%
}\frac{(1-\left\vert w\right\vert ^{2})^{b}}{\left\vert 1-\overline
{w}z\right\vert ^{2+a+b}}f\left(  w\right)  dw.
\end{align*}
Suppose $t\in\mathbb{R}$ and $1\leq p<\infty.$Then $T$ is bounded on
$L^{p}\left(  \mathbb{D},(1-\left\vert z\right\vert ^{2})^{t}dA\right)  $ if
and only if $S$ is bounded on $L^{p}\left(  \mathbb{D},(1-\left\vert
z\right\vert ^{2})^{t}dA\right)  $ if and only if%
\begin{equation}
-pa<t+1<p\left(  b+1\right)  .\label{paramcond}%
\end{equation}

\end{corollary}

We now use Proposition \ref{Schurtest} to show that if $\mathcal{A}$,
$\mathcal{B\subset T}\ $are well separated then a certain bilinear operator
mapping on $\ell^{2}\left(  \mathcal{A}\right)  \times\ell^{2}\left(
\mathcal{B}\right)  $ maps boundedly into $L^{2}\left(  \mathbb{D}\right)  .$

\begin{lemma}
\label{bilin} Suppose $\mathcal{A}$ and $\mathcal{B}$ are subsets of $T$,
$h\in\ell^{2}\left(  \mathcal{A}\right)  $ and $k\in\ell^{2}\left(
\mathcal{B}\right)  ,$ and $1/2<\alpha<1.$ Suppose further that $\mathcal{A}$
and $\mathcal{B}$ satisfy the separation condition: $\forall\kappa
\in\mathcal{A}$, $\gamma\in$ $\mathcal{B}$ we have%
\begin{equation}
\left\vert \kappa-\gamma\right\vert \geq(1-\left\vert \gamma\right\vert
^{2})^{\alpha}.\label{sep}%
\end{equation}
Then the bilinear map of $(h,k)$ to functions on the disk given by
\[
T\left(  h,b\right)  \left(  z\right)  =\left(  \sum_{\kappa\in\mathcal{A}%
}h\left(  \kappa\right)  \frac{(1-\left\vert \kappa\right\vert ^{2})^{1+s}%
}{\left\vert 1-\overline{\kappa}z\right\vert ^{2+s}}\right)  \left(
\sum_{\gamma\in\mathcal{B}}b\left(  \gamma\right)  \frac{(1-\left\vert
\gamma\right\vert ^{2})^{1+s}}{\left\vert 1-\overline{\gamma}z\right\vert
^{1+s}}\right)
\]
is bounded from $\ell^{2}\left(  \mathcal{A}\right)  \times\ell^{2}\left(
\mathcal{B}\right)  $ to $L^{2}\left(  \mathbb{D}\right)  $.
\end{lemma}

\begin{remark}
For $h\in\ell^{2}\left(  \mathcal{A}\right)  $ and $b\in\ell^{2}\left(
\mathcal{B}\right)  $ set
\[
H\left(  z\right)  =\sum_{\kappa\in\mathcal{A}}h\left(  \kappa\right)
\frac{(1-\left\vert \kappa\right\vert ^{2})^{1+s}}{\left(  1-\overline{\kappa
}z\right)  ^{2+s}}\text{, }B(z)=\sum_{\gamma\in\mathcal{B}}b(\gamma
)\frac{(1-\left\vert \gamma\right\vert ^{2})^{1+s}}{(1-\overline{\gamma
}z)^{1+s}}.
\]
By \cite[Thm 2.30]{Zhu} $H\in L^{2}\left(  \mathbb{D}\right)  $ and
$B\in\mathcal{D}$. There are unbounded functions in $\mathcal{D}$ hence these
facts do not insure show $HB\in L^{2}\left(  \mathbb{D}\right)  $.\ The lemma
shows that if $\mathcal{A}$ and $\mathcal{B}$ are separated then $HB\in
L^{2}\left(  \mathbb{D}\right)  .$

\begin{proof}
We will verify the hypotheses of the previous proposition. The kernel function
here is%
\[
H\left(  z,\kappa,\gamma\right)  =\frac{(1-\left\vert \kappa\right\vert
^{2})^{1+s}}{\left\vert 1-\overline{\kappa}z\right\vert ^{2+s}}\frac
{(1-\left\vert \gamma\right\vert ^{2})^{1+s}}{\left\vert 1-\overline{\gamma
}z\right\vert ^{1+s}},\ z\in\mathbb{D},\kappa\in\mathcal{A},\gamma
\in\mathcal{B},
\]
with Lebesgue measure on $\mathbb{D}$, and counting measure on $\mathcal{A}$
and $\mathcal{B}$. We will take as Schur functions%
\[
h\left(  z\right)  =(1-\left\vert z\right\vert ^{2})^{-1/4}\text{, }k\left(
\kappa\right)  =(1-\left\vert \kappa\right\vert ^{2})^{1/4}\text{ and
}m\left(  \gamma\right)  =(1-\left\vert \gamma\right\vert ^{2})^{\varepsilon
/2},
\]
on $\mathbb{D}$, $\mathcal{A}$ and $\mathcal{B}$\ respectively, where
$\varepsilon=\varepsilon(\alpha,s)>0$ will be chosen sufficiently small later.
We must then verify%
\begin{equation}
\sum_{\kappa\in\mathcal{A}}\sum_{\gamma\in\mathcal{B}}\frac{(1-\left\vert
\kappa\right\vert ^{2})^{3/2+s}}{\left\vert 1-\overline{\kappa}z\right\vert
^{2+s}}\frac{(1-\left\vert \gamma\right\vert ^{2})^{1+\varepsilon+s}%
}{\left\vert 1-\overline{\gamma}z\right\vert ^{1+s}}\leq A^{2}(1-\left\vert
z\right\vert ^{2})^{-1/2},\label{Schur1}%
\end{equation}
for $z\in\mathbb{D}$, and%
\begin{equation}
\int_{\mathbb{D}}\frac{(1-\left\vert \kappa\right\vert ^{2})^{1+s}}{\left\vert
1-\overline{\kappa}z\right\vert ^{2+s}}\frac{(1-\left\vert \gamma\right\vert
^{2})^{1+s}}{\left\vert 1-\overline{\gamma}z\right\vert ^{1+s}}(1-\left\vert
z\right\vert ^{2})^{-1/2}dA\leq B^{2}(1-\left\vert \kappa\right\vert
^{2})^{1/2}(1-\left\vert \gamma\right\vert ^{2})^{\varepsilon},\label{Schur2}%
\end{equation}
for $\kappa\in\mathcal{A}$ and $\gamma\in\mathcal{B}$.
\end{proof}
\end{remark}

\begin{lemma}
\begin{proof}
To prove (\ref{Schur1}) we write%
\begin{align*}
& \sum_{\kappa\in\mathcal{A}}\sum_{\gamma\in\mathcal{B}}\frac{(1-\left\vert
\kappa\right\vert ^{2})^{3/2+s}}{\left\vert 1-\overline{\kappa}z\right\vert
^{2+s}}\frac{(1-\left\vert \gamma\right\vert ^{2})^{1+\varepsilon+s}%
}{\left\vert 1-\overline{\gamma}z\right\vert ^{1+s}}=\\
& \qquad\qquad\qquad\qquad\left(  \sum_{\kappa\in\mathcal{A}}\frac
{(1-\left\vert \kappa\right\vert ^{2})^{3/2+s}}{\left\vert 1-\overline{\kappa
}z\right\vert ^{2+s}}\right)  \left(  \sum_{\gamma\in\mathcal{B}}%
\frac{(1-\left\vert \gamma\right\vert ^{2})^{1+\varepsilon+s}}{\left\vert
1-\overline{\gamma}z\right\vert ^{1+s}}\right)  .
\end{align*}
Then from (\ref{intest}) we obtain%
\[
\sum_{\kappa\in\mathcal{A}}\frac{(1-\left\vert \kappa\right\vert ^{2}%
)^{3/2+s}}{\left\vert 1-\overline{\kappa}z\right\vert ^{2+s}}\leq
C\int_{\mathbb{D}}\frac{(1-\left\vert w\right\vert ^{_{2}})^{-1/2+s}%
}{\left\vert 1-\overline{w}z\right\vert ^{2+s}}dw\leq C(1-\left\vert
z\right\vert ^{2})^{-1/2}%
\]
and%
\[
\sum_{\gamma\in\mathcal{B}}\frac{(1-\left\vert \gamma\right\vert
^{2})^{1+\varepsilon+s}}{\left\vert 1-\overline{\gamma}z\right\vert ^{1+s}%
}\leq C\int_{\zeta\in V_{G}}\frac{(1-\left\vert \zeta\right\vert
^{2})^{-1+\varepsilon+s}}{\left\vert 1-\overline{\zeta}z\right\vert ^{1+s}%
}dA\leq C,
\]
which yields (\ref{Schur1}).

We now prove (\ref{Schur2}) We will make repeated use\ of (\ref{sep}) as well
as its consequence via the triangle inequality: $\forall\kappa\in\mathcal{A}$,
$\gamma\in$ $\mathcal{B}$ $(1-\left\vert \kappa\right\vert ^{2})\leq
C\left\vert \kappa-\gamma\right\vert .$ We set $\kappa^{\ast}=\kappa
/\left\vert \kappa\right\vert ,$ $\gamma^{\ast}=\gamma/\left\vert
\gamma\right\vert .$%
\begin{align*}
& \int_{\mathbb{D}}\frac{(1-\left\vert \kappa\right\vert ^{2})^{1+s}%
}{\left\vert 1-\overline{\kappa}z\right\vert ^{2+s}}\frac{(1-\left\vert
\gamma\right\vert ^{2})^{1+s}}{\left\vert 1-\overline{\gamma}z\right\vert
^{1+s}}(1-\left\vert z\right\vert ^{2})^{-1/2}dA\\
& =\int\limits_{\left\vert z-\gamma^{\ast}\right\vert \leq1-\left\vert
\gamma\right\vert ^{2}}+\int\limits_{1-\left\vert \gamma\right\vert ^{2}%
\leq\left\vert z-\gamma^{\ast}\right\vert \leq\frac{1}{2}\left\vert
\kappa-\gamma\right\vert }\\
& +\int\limits_{\left\vert z-\kappa^{\ast}\right\vert \leq1-\left\vert
\kappa\right\vert ^{2}}+\int\limits_{1-\left\vert \kappa\right\vert ^{2}%
\leq\left\vert z-\kappa^{\ast}\right\vert \leq\frac{1}{2}\left\vert
\kappa-\gamma\right\vert }+\int\limits_{\left\vert z-\gamma^{\ast}\right\vert
,\left\vert z-\kappa^{\ast}\right\vert \geq\left\vert \kappa-\gamma\right\vert
}...dA\\
& =I+II+III+IV+V.
\end{align*}
We have%
\begin{align*}
I  & \approx\frac{(1-\left\vert \kappa\right\vert ^{2})^{1+s}}{\left\vert
\kappa-\gamma\right\vert ^{2+s}}\int_{\left\vert z-\gamma^{\ast}\right\vert
\leq1-\left\vert \gamma\right\vert ^{2}}(1-\left\vert z\right\vert
^{2})^{-1/2}dA\\
& \approx\frac{(1-\left\vert \kappa\right\vert ^{2})^{1+s}(1-\left\vert
\gamma\right\vert ^{2})^{3/2}}{\left\vert \kappa-\gamma\right\vert ^{2+s}}\leq
C(1-\left\vert \kappa\right\vert ^{2})^{1/2}(1-\left\vert \gamma\right\vert
^{2})^{3\left(  1-\alpha\right)  /2}.
\end{align*}
Similarly we have%
\begin{align*}
II  & \approx\frac{(1-\left\vert \kappa\right\vert ^{2})^{1+s}(1-\left\vert
\gamma\right\vert ^{2})^{1+s}}{\left\vert \kappa-\gamma\right\vert ^{2+s}}%
\int\limits_{1-\left\vert \gamma\right\vert ^{2}\leq\left\vert z-\gamma^{\ast
}\right\vert \leq\frac{1}{2}\left\vert \kappa-\gamma\right\vert }%
\frac{(1-\left\vert z\right\vert ^{2})^{-1/2}}{\left\vert z-\gamma^{\ast
}\right\vert ^{1+s}}dA\\
& \approx\frac{(1-\left\vert \kappa\right\vert ^{2})^{1+s}(1-\left\vert
\gamma\right\vert ^{2})^{1+s}}{\left\vert \kappa-\gamma\right\vert ^{2+s}%
}(1-\left\vert \gamma\right\vert ^{2})^{1/2-s}\\
& =\frac{(1-\left\vert \kappa\right\vert ^{2})^{1+s}(1-\left\vert
\gamma\right\vert ^{2})^{3/2}}{\left\vert \kappa-\gamma\right\vert ^{2+s}}\leq
C(1-\left\vert \kappa\right\vert ^{2})^{1/2}(1-\left\vert \gamma\right\vert
^{2})^{3\left(  1-\alpha\right)  /2}.
\end{align*}
Continuing we obtain%
\[
III\approx\frac{(1-\left\vert \kappa\right\vert ^{2})^{1/2}(1-\left\vert
\gamma\right\vert ^{2})^{1+s}}{\left\vert \kappa-\gamma\right\vert ^{1+s}}\leq
C(1-\left\vert \kappa\right\vert ^{2})^{1/2}(1-\left\vert \gamma\right\vert
^{2})^{\left(  1+s\right)  \left(  1-\alpha\right)  },
\]
and similarly,%
\[
IV\leq C(1-\left\vert \kappa\right\vert ^{2})^{1/2}(1-\left\vert
\gamma\right\vert ^{2})^{\varepsilon},
\]
for some $\varepsilon>0$. Finally%
\begin{align*}
V  & \approx\int\limits_{\left\vert z-\gamma^{\ast}\right\vert ,\left\vert
z-\kappa^{\ast}\right\vert \geq\left\vert \kappa-\gamma\right\vert }%
\frac{(1-\left\vert \kappa\right\vert ^{2})^{1+s}}{\left\vert z-\kappa^{\ast
}\right\vert ^{2+s}}\frac{(1-\left\vert \gamma\right\vert ^{2})^{1+s}%
}{\left\vert z-\gamma^{\ast}\right\vert ^{1+s}}(1-\left\vert z\right\vert
^{2})^{-1/2}dA\\
& \approx\frac{(1-\left\vert \kappa\right\vert ^{2})^{1+s}(1-\left\vert
\gamma\right\vert ^{2})^{1+s}}{\left\vert \kappa-\gamma\right\vert ^{3/2+2s}%
}\\
& \leq C(1-\left\vert \kappa\right\vert ^{2})^{1/2}(1-\left\vert
\gamma\right\vert ^{2})^{\left(  1+s\right)  \left(  1-\alpha\right)  }.
\end{align*}

\end{proof}
\end{lemma}

\section{The Main Bilinear Estimate}

To complete the proof we will show that $\mu_{b}$ is a $\mathcal{D}$-Carleson
measure by verifying Stegenga's condition (\ref{Steg}); that is, we will show
that for any finite collection of disjoint arcs $\left\{  I_{j}\right\}
_{j=1}^{N}$ in the circle $\mathbb{T}$ we have
\[
\mu_{b}\left(  \bigcup\nolimits_{j=1}^{N}T\left(  I_{j}\right)  \right)  \leq
C\ Cap_{\mathbb{D}}\left(  \bigcup\nolimits_{j=1}^{N}I_{j}\right)  .
\]
In fact we will see that it suffices to verify this for the sets $G=\cup
_{j=1}^{N}I_{j}$ described in (\ref{defGdelta}) that are almost extremal for
(\ref{Gmax}). We will prove the inequality%
\begin{equation}
\mu_{b}\left(  V_{G}\right)  \leq C\left\Vert T_{b}\right\Vert ^{2}%
Cap_{\mathbb{D}}\left(  G\right)  .\label{maintheta}%
\end{equation}
Once we have this Corollary \ref{capequiv} yields%
\[
M=\frac{\int_{\mathbb{T}}\mu_{b}\left(  T_{\theta}\left(  G\right)  \right)
d\theta}{\int_{\mathbb{T}}Cap_{\theta}\left(  G\right)  d\theta}\leq\frac
{\mu_{b}\left(  V_{G}\right)  }{\int_{\mathbb{T}}Cap_{\theta}\left(  G\right)
d\theta}\leq C\left\Vert T_{b}\right\Vert ^{2}.
\]
By Corollary \ref{M} $\left\Vert \mu_{b}\right\Vert _{\mathcal{D}%
-Carleson}^{2}\approx M\ $which then completes the proof of Theorem \ref{main}.

We now turn to (\ref{maintheta}). Let $1/2<\beta<\beta_{1}<\gamma<\alpha<1$
with additional constraints to be added later. Suppose $G$ (\ref{defGdelta})
with $\varepsilon>0$ to be chosen below. We define in succession the following
regions in the disk,%
\begin{align*}
V_{G}  & =T_{\mathcal{T}}\left(  G\right)  ,\\
V_{G}^{\alpha}  & =G_{\mathbb{D}}^{\alpha},\\
V_{G}^{\gamma}  & =\widehat{\left(  V_{G}^{\alpha}\right)  _{\mathcal{T}%
}^{\gamma/\alpha}},\\
V_{G}^{\beta}  & =\left(  V_{G}^{\gamma}\right)  _{\mathbb{D}}^{\beta/\gamma}.
\end{align*}
Thus $V_{G}$ is the $\mathcal{T}$-tent associated with $G$, $V_{G}^{\alpha}$
is a disk blowup of $G$, $V_{G}^{\gamma}$ is a $\mathcal{T}$-capacitary blowup
of $V_{G}^{\alpha}$, and $V_{G}^{\beta}$ is a disk blowup of $V_{G}^{\gamma}$.
Using the natural bijections described earlier, we write%
\begin{equation}
V_{G}=\left\{  w_{k}\right\}  _{k}\text{ and }V_{G}^{\alpha}=\left\{
w_{k}^{\alpha}\right\}  _{k}\text{ and }V_{G}^{\gamma}=\left\{  w_{k}^{\gamma
}\right\}  _{k}\text{ and }V_{G}^{\beta}=\{w_{k}^{\beta}\}_{k}%
,\label{def(E,F)}%
\end{equation}
with $w_{k},$ $w_{k}^{\alpha},$ $w_{k}^{\gamma},$ $w_{k}^{\beta}\in
\mathcal{T}$. Following earlier notation we write $E=V_{G}^{\alpha}$ and
$F=V_{G}^{\gamma}$.

We proceed by estimating $T_{b}(f,g)$ for well chosen $f\ $and $g$ in
$\mathcal{D}.$ Let $\Phi$ be as in (\ref{defFw}); we then have the estimates
in Proposition \ref{newBoe} and Corollary \ref{smallcap}. Set $g=\Phi^{2};$ in
particular note that $g$ is approximately equal to $\chi_{V_{G}}.$ The
function $f$ will be, approximately, $b^{\prime}\chi_{V_{G}};$
\begin{equation}
f\left(  z\right)  =\Gamma_{s}\left(  \frac{1}{\left(  1+s\right)
\overline{\zeta}}\chi_{V_{G}}b^{\prime}\left(  \zeta\right)  \right)  \left(
z\right)  =\int\limits_{V_{G}}\frac{b^{\prime}\left(  \zeta\right)
(1-\left\vert \zeta\right\vert ^{2})^{s}}{\left(  1-\overline{\zeta}z\right)
^{1+s}}\frac{dA}{\left(  1+s\right)  \overline{\zeta}}.\label{newf}%
\end{equation}

We now analyze $T_{b}(f,g).$ From (\ref{newf}) and (\ref{repro}) we have%
\begin{align*}
f^{\prime}\left(  z\right)   & =\int_{V_{G}}\frac{b^{\prime}\left(
\zeta\right)  (1-\left\vert \zeta\right\vert ^{2})^{s}}{\left(  1-\overline
{\zeta}z\right)  ^{2+s}}dA\\
& =b^{\prime}\left(  z\right)  -\int_{\mathbb{D}\setminus V_{G}}%
\frac{b^{\prime}\left(  \zeta\right)  (1-\left\vert \zeta\right\vert ^{2}%
)^{s}}{\left(  1-\overline{\zeta}z\right)  ^{2+s}}dA\\
& =b^{\prime}\left(  z\right)  +\Lambda b^{\prime}\left(  z\right)  ,
\end{align*}
where the last term is defined by%
\begin{equation}
\Lambda b^{\prime}\left(  z\right)  =-\int_{\mathbb{D}\setminus V_{G}}%
\frac{b^{\prime}\left(  \zeta\right)  \left(  1-\left\vert \zeta\right\vert
\right)  ^{s}}{\left(  1-\overline{\zeta}z\right)  ^{2+s}}dA.\label{defLambda}%
\end{equation}
We have%

\begin{align}
T_{b}\left(  f,g\right)   & =\left(  f\Phi^{2}\bar{b}\right)  (0)+\int
_{\mathbb{D}}\left\{  f^{\prime}\left(  z\right)  \Phi\left(  z\right)
+2f\left(  z\right)  \Phi^{\prime}\left(  z\right)  \right\}  \Phi\left(
z\right)  \overline{b^{\prime}\left(  z\right)  }dA\label{rr1}\\
& =\left(  f\Phi^{2}\bar{b}\right)  (0)+\int_{\mathbb{D}}\left\vert b^{\prime
}\left(  z\right)  \right\vert ^{2}\Phi\left(  z\right)  ^{2}dA\nonumber\\
& +2\int_{\mathbb{D}}\Phi\left(  z\right)  \Phi^{\prime}\left(  z\right)
f\left(  z\right)  \overline{b^{\prime}\left(  z\right)  }dA+\int_{\mathbb{D}%
}\Lambda b^{\prime}\left(  z\right)  \overline{b^{\prime}\left(  z\right)
}\Phi\left(  z\right)  ^{2}dA\nonumber\\
& =(1)+(2)+(3)+(4).\nonumber
\end{align}

Now we write%
\begin{align}
(2)  & =\int_{\mathbb{D}}\left\vert b^{\prime}\left(  z\right)  \right\vert
^{2}\Phi\left(  z\right)  ^{2}dA\label{I}\\
& =\left\{  \int_{V_{G}}+\int_{V_{G}^{\beta}\setminus V_{G}}+\int
_{\mathbb{D}\setminus V_{G}^{\beta}}\right\}  \left\vert b^{\prime}\left(
z\right)  \right\vert ^{2}\Phi\left(  z\right)  ^{2}dA\nonumber\\
& =(2_{A})+(2_{B})+(2_{C}).\nonumber
\end{align}
The main term is (2$_{A}$). By (\ref{estimates}) and (\ref{Dnorm}) it
satisfies%
\begin{align}
(2_{A})  & =\mu_{b}\left(  V_{G}\right)  +\int_{V_{G}}\left\vert b^{\prime
}\left(  z\right)  \right\vert ^{2}(\Phi\left(  z\right)  ^{2}-1)dA\label{IV}%
\\
& =\mu_{b}\left(  V_{G}\right)  +O(\left\Vert T_{b}\right\Vert ^{2}%
Cap_{\mathcal{T}}\left(  E,F\right)  ),\nonumber
\end{align}
Rearranging this and using (\ref{rr1}) and (\ref{I}) we find%
\begin{equation}
\mu_{b}\left(  V_{G}\right)  \leq C\left\Vert T_{b}\right\Vert ^{2}%
Cap_{\mathcal{T}}\left(  E,F\right)  +\left\vert T_{b}(f,g)\right\vert
+\left\vert (1)\right\vert +(2_{B})+(2_{C})+\left\vert (3)\right\vert
+\left\vert (4)\right\vert \label{rr2}%
\end{equation}

Using the boundedness of $T_{b}$ and Corollary \ref{smallcap} we have
\begin{align}
\left\vert T_{b}(f,g)\right\vert  & =\left\vert T_{b}(f,\Phi^{2})\right\vert
=\left\vert T_{b}(f\Phi,\Phi)\right\vert \label{norm}\\
& \leq\left\Vert T_{b}\right\Vert \left\Vert f\Phi\right\Vert _{\mathcal{D}%
}\left\Vert \Phi\right\Vert _{\mathcal{D}}\leq C\left\Vert T_{b}\right\Vert
\left\Vert f\Phi\right\Vert _{\mathcal{D}}\sqrt{Cap_{\mathcal{T}}\left(
E,F\right)  }.\nonumber
\end{align}
For $(1)$ we use the elementary estimate%
\[
\left\vert (1)\right\vert \leq C\left\Vert b\right\Vert _{\mathcal{D}}%
^{2}Cap_{\mathcal{T}}\left(  E,F\right)  \leq C\left\Vert T_{b}\right\Vert
^{2}Cap_{\mathcal{T}}\left(  E,F\right)  .
\]
For (2$_{B}$) we use (\ref{inpart}) to obtain%
\begin{equation}
(2_{B})\leq C\mu_{b}\left(  V_{G}^{\beta}\setminus V_{G}\right)  \leq
C\varepsilon\mu_{b}\left(  V_{G}\right)  .\label{V}%
\end{equation}
Using (\ref{estimates}) once more, we see that (2$_{C}$) satisfies%
\begin{equation}
(2_{C})\leq\int_{\mathbb{D}\setminus V_{G}^{\beta}}\left\vert b^{\prime
}\left(  z\right)  \right\vert ^{2}\left(  C_{\alpha,\beta,\rho}%
Cap_{\mathcal{T}}\left(  E,F\right)  \right)  ^{2}dA\leq C\left\Vert
T_{b}\right\Vert ^{2}Cap_{\mathcal{T}}\left(  E,F\right)  .\label{VI}%
\end{equation}
Putting these estimates into (\ref{rr2}) we obtain
\begin{equation}
\mu_{b}\left(  V_{G}\right)  \leq C\left(  \left\Vert T_{b}\right\Vert
^{2}Cap_{\mathcal{T}}\left(  E,F\right)  +\left\Vert T_{b}\right\Vert
\left\Vert f\Phi\right\Vert _{\mathcal{D}}\sqrt{Cap_{\mathcal{T}}\left(
E,F\right)  }+\left\vert (3)\right\vert +\left\vert (4)\right\vert \right)
.\label{rr3}%
\end{equation}

For small positive $\varepsilon$ we estimate (3) using Cauchy-Schwarz as
follows:%
\begin{align*}
\left\vert (3)\right\vert  & \leq2\int_{\mathbb{D}}\left\vert \Phi\left(
z\right)  b^{\prime}\left(  z\right)  \right\vert \left\vert \Phi^{\prime
}\left(  z\right)  f\left(  z\right)  \right\vert dA\\
& \leq\varepsilon\int_{\mathbb{D}}\left\vert \Phi\left(  z\right)  b^{\prime
}\left(  z\right)  \right\vert ^{2}dA+\frac{C}{\varepsilon}\int_{\mathbb{D}%
}\left\vert \Phi^{\prime}\left(  z\right)  f\left(  z\right)  \right\vert
^{2}dA\\
& =(3_{A})+(3_{B}).
\end{align*}
Using the decomposition and the argument surrounding term (2) we obtain%
\begin{align}
(3_{A})  & \leq\varepsilon\left\{  \int_{V_{G}}+\int_{V_{G}^{\beta}\setminus
V_{G}}+\int_{\mathbb{D}\setminus V_{G}^{\beta}}\right\}  \left\vert
\Phi\left(  z\right)  b^{\prime}\left(  z\right)  \right\vert ^{2}%
dA\label{VII}\\
& \leq C\varepsilon\left(  \mu_{b}\left(  V_{G}\right)  +C\left\Vert
T_{b}\right\Vert ^{2}Cap_{\mathcal{T}}\left(  E,F\right)  \right)  .\nonumber
\end{align}
To estimate term $(3_{B})$ we use%
\begin{align*}
\left\vert f\left(  z\right)  \right\vert  & \leq\left\vert \Gamma_{s}\left(
\frac{1}{\left(  1+s\right)  \overline{\zeta}}\chi_{V_{G}}b^{\prime}\left(
\zeta\right)  \right)  \left(  z\right)  \right\vert \\
& \leq\int_{V_{G}}\frac{\left(  1-\left\vert \zeta\right\vert ^{2}\right)
^{s}}{\left\vert 1-\overline{\zeta}z\right\vert ^{1+s}}\left\vert b^{\prime
}\left(  \zeta\right)  \right\vert dA\\
& \approx\sum_{\gamma\in\mathcal{T}\cap V_{G}}\frac{(1-\left\vert
\gamma\right\vert ^{2})^{1+s}}{\left\vert 1-\overline{\gamma}z\right\vert
^{1+s}}\int_{B_{\gamma}}\left\vert b^{\prime}\left(  \zeta\right)  \right\vert
\left(  1-\left\vert \zeta\right\vert ^{2}\right)  d\lambda\left(
\zeta\right) \\
& =\sum_{\gamma\in\mathcal{T}\cap V_{G}}\frac{(1-\left\vert \gamma\right\vert
^{2})^{1+s}}{\left\vert 1-\overline{\gamma}z\right\vert ^{1+s}}b\left(
\gamma\right)  ,
\end{align*}
where%
\[
\sum_{\gamma\in\mathcal{T}\cap V_{G}}b\left(  \gamma\right)  ^{2}\approx
\sum_{\gamma\in\mathcal{T}\cap V_{G}}\int_{B_{\gamma}}\left\vert b^{\prime
}\left(  \zeta\right)  \right\vert ^{2}(1-\left\vert \zeta\right\vert
^{2})^{2}d\lambda\left(  \zeta\right)  =\int_{V_{G}}\left\vert b^{\prime
}\left(  \zeta\right)  \right\vert ^{2}dA.
\]

We now use the separation of $\mathbb{D}\setminus V_{G}^{\alpha}$ and $V_{G}.$
The facts that $\mathcal{A}=supp\left(  h\right)  \subset\mathbb{D}\setminus
V_{G}^{\alpha}$ and $\mathcal{B}=\mathcal{T}\cap V_{G}\subset V_{G}$, together
with Lemma \ref{geoseparation}, insure that (\ref{sep}) is satisfied and hence
we can use Lemma \ref{bilin} and the representation of $\Phi$ in (\ref{defFw})
to continue with
\[
(3_{B})=\int_{\mathbb{D}}\left\vert \Phi^{\prime}\left(  z\right)  f\left(
z\right)  \right\vert ^{2}dA\leq C\left(  \sum_{\kappa\in\mathcal{A}}h\left(
\kappa\right)  ^{2}\right)  \left(  \sum_{\gamma\in\mathcal{B}}b\left(
\gamma\right)  ^{2}\right)  .
\]
We also have from (\ref{Dnorm}) and Corollary \ref{smallcap} that
\[
\left(  \sum_{\kappa\in\mathcal{A}}h\left(  \kappa\right)  ^{2}\right)
\left(  \sum_{\gamma\in\mathcal{B}}b\left(  \gamma\right)  ^{2}\right)  \leq
CCap_{\mathcal{T}}\left(  E,F\right)  \left\Vert T_{b}\right\Vert ^{2}.
\]
Altogether we then have%
\begin{equation}
(3_{B})\leq C\ Cap_{\mathcal{T}}\left(  E,F\right)  \left\Vert T_{b}%
\right\Vert ^{2},\label{VIII}%
\end{equation}
and thus also%
\begin{equation}
\left\vert (3)\right\vert \leq\varepsilon\mu_{b}\left(  V_{G}\right)
+C\left\Vert T_{b}\right\Vert ^{2}Cap_{\mathcal{T}}\left(  E,F\right)
.\label{II}%
\end{equation}

We begin our estimate of term (4) by%
\begin{align}
\left\vert (4)\right\vert  & =\left\vert \int_{\mathbb{D}}\Lambda b^{\prime
}\left(  z\right)  \overline{b^{\prime}\left(  z\right)  }\Phi\left(
z\right)  ^{2}dA\right\vert \label{rr4}\\
& \leq\sqrt{\int_{\mathbb{D}}\left\vert b^{\prime}\left(  z\right)
\Phi\left(  z\right)  \right\vert ^{2}dA}\sqrt{\int_{\mathbb{D}}\left\vert
\Lambda b^{\prime}\left(  z\right)  \Phi\left(  z\right)  \right\vert ^{2}%
dA}.\nonumber
\end{align}
where the first factor is $\sqrt{\left(  3_{A}\right)  /\varepsilon}$. We
claim the following estimate for the second factor $\sqrt{(4_{A})}:=\left\Vert
\Phi\Lambda b^{\prime}\right\Vert _{L^{2}\left(  \mathbb{D}\right)  }$:

\begin{lemma}%
\begin{equation}
(4_{A})=\int_{\mathbb{D}}\left\vert \Phi\left(  z\right)  \Lambda b^{\prime
}\left(  z\right)  \right\vert ^{2}dA\leq C\mu_{b}\left(  V_{G}^{\beta
}\setminus V_{G}\right)  +C\left\Vert T_{b}\right\Vert ^{2}Cap_{\mathcal{T}%
}\left(  E,F\right) \label{pres}%
\end{equation}

\begin{proof}
From (\ref{defLambda}) we obtain%
\begin{align*}
(4_{A})  & =\int_{\mathbb{D}}\left\vert \Phi\left(  z\right)  \right\vert
^{2}\left\vert \left\{  \int_{V_{G}^{\beta}\setminus V_{G}}+\int
_{\mathbb{D}\setminus V_{G}^{\beta}}\right\}  \frac{b^{\prime}\left(
\zeta\right)  \left(  1-\left\vert \zeta\right\vert \right)  ^{s}}{\left(
1-\overline{\zeta}z\right)  ^{2+s}}dA\right\vert ^{2}dA\\
& \leq C\int_{\mathbb{D}}\left\vert \Phi\left(  z\right)  \right\vert
^{2}\left(  \int_{V_{G}^{\beta}\setminus V_{G}}\frac{\left\vert b^{\prime
}\left(  \zeta\right)  \right\vert \left(  1-\left\vert \zeta\right\vert
\right)  ^{s}}{\left\vert 1-\overline{\zeta}z\right\vert ^{2+s}}dA\right)
^{2}dA\\
& \qquad\qquad+C\int_{\mathbb{D}}\left\vert \Phi\left(  z\right)  \right\vert
^{2}\left\vert \int_{\mathbb{D}\setminus V_{G}^{\beta}}\frac{b^{\prime}\left(
\zeta\right)  \left(  1-\left\vert \zeta\right\vert \right)  ^{s}}{\left(
1-\overline{\zeta}z\right)  ^{2+s}}dA\right\vert ^{2}dA\\
& =(4_{AA})+(4_{AB}).
\end{align*}
Corollary \ref{210}\ shows that%
\begin{align*}
\left\vert (4_{AA})\right\vert  & \leq\int_{\mathbb{D}}\left(  \int
_{V_{G}^{\beta}\setminus V_{G}}\frac{\left\vert b^{\prime}\left(
\zeta\right)  \right\vert \left(  1-\left\vert \zeta\right\vert \right)  ^{s}%
}{\left\vert 1-\overline{\zeta}z\right\vert ^{2+s}}dA\right)  ^{2}dA\\
& \leq C\int_{V_{G}^{\beta}\setminus V_{G}}\left\vert b^{\prime}\left(
\zeta\right)  \right\vert ^{2}dA=C\mu_{b}\left(  V_{G}^{\beta}\setminus
V_{G}\right)  .
\end{align*}
We write the second integral as%
\begin{align*}
(4_{AB})  & =\left\{  \int_{V_{G}^{\gamma}}+\int_{\mathbb{D}\setminus
V_{G}^{\gamma}}\right\}  \left\vert \Phi\left(  z\right)  \right\vert
^{2}\left\vert \int_{\mathbb{D}\setminus V_{G}^{\beta}}\frac{b^{\prime}\left(
\zeta\right)  \left(  1-\left\vert \zeta\right\vert \right)  ^{s}}{\left(
1-\overline{\zeta}z\right)  ^{2+s}}dA\right\vert ^{2}dA\\
& =(4_{ABA})+(4_{ABB}),
\end{align*}
where by Corollary \ref{210} again,%
\begin{align*}
(4_{ABB})  & \leq C\ Cap_{\mathcal{T}}\left(  E,F\right)  ^{2}\int
_{\mathbb{D}}\left\vert b^{\prime}\left(  \zeta\right)  \right\vert ^{2}dA\\
& \leq C\left\Vert T_{b}\right\Vert ^{2}Cap_{\mathcal{T}}\left(  E,F\right)
^{2}\\
& \leq C\left\Vert T_{b}\right\Vert ^{2}Cap_{\mathcal{T}}\left(  E,F\right)  .
\end{align*}
Finally, with $\beta<\beta_{1}<\gamma<\alpha<1$, Corollary \ref{210} shows
that the term $(4_{ABA})$ satisfies the following estimate. Recall that
$V_{G}^{\gamma}=\cup J_{k}^{\gamma}$ and $w_{j}^{\gamma}=z\left(
J_{k}^{\gamma}\right)  $. We set $A_{\ell}=\left\{  k:J_{k}^{\gamma}\subset
J_{\ell}^{\beta_{1}}\right\}  $ and define $\ell\left(  k\right)  $ by the
condition $k\in A_{\ell\left(  k\right)  }$. From Lemma \ref{geoseparation} we
have that, with $\rho=\beta_{1}/\gamma,$ $\operatorname{sidelength}$%
($J_{k}^{\gamma})\leq\operatorname{sidelength}(J_{\ell}^{\beta_{1}})^{1/\rho}%
$. Hence
\begin{align*}
\left(  4_{ABA}\right)   & \leq C\int_{V_{G}^{\gamma}}\left(  \int
_{\mathbb{D}\setminus V_{G}^{\beta}}\frac{\left\vert b^{\prime}\left(
\zeta\right)  \right\vert \left(  1-\left\vert \zeta\right\vert \right)  ^{s}%
}{\left\vert 1-\overline{\zeta}z\right\vert ^{2+s}}d\zeta\right)  ^{2}dA\\
& \approx C\sum_{k}\int_{J_{k}^{\gamma}}\left\vert J_{k}^{\gamma}\right\vert
\left(  \int_{\mathbb{D}\setminus V_{G}^{\beta}}\frac{\left\vert b^{\prime
}\left(  \zeta\right)  \right\vert \left(  1-\left\vert \zeta\right\vert
\right)  ^{s}}{\left\vert 1-\overline{\zeta}w_{k}^{\gamma}\right\vert ^{2+s}%
}d\zeta\right)  ^{2}dA\\
& =C\sum_{k}\frac{\left\vert J_{k}^{\gamma}\right\vert }{\left\vert
J_{\ell\left(  k\right)  }^{\beta_{1}}\right\vert }\left\vert J_{\ell\left(
k\right)  }^{\beta_{1}}\right\vert \left(  \int_{\mathbb{D}\setminus
V_{G}^{\beta}}\frac{\left\vert b^{\prime}\left(  \zeta\right)  \right\vert
\left(  1-\left\vert \zeta\right\vert \right)  ^{s}}{\left\vert 1-\overline
{\zeta}w_{k}^{\gamma}\right\vert ^{2+s}}d\zeta\right)  ^{2}%
\end{align*}%
\begin{align*}
\qquad\qquad\qquad & \approx C\sum_{\ell}\frac{\sum_{k\in A_{\ell}}\left\vert
J_{k}^{\gamma}\right\vert }{\left\vert J_{\ell}^{\beta_{1}}\right\vert }%
\int_{J_{\ell}^{\beta_{1}}}\left(  \int_{\mathbb{D}\setminus V_{G}^{\beta}%
}\frac{\left\vert b^{\prime}\left(  \zeta\right)  \right\vert \left(
1-\left\vert \zeta\right\vert \right)  ^{s}}{\left\vert 1-\overline{\zeta
}z\right\vert ^{2+s}}d\zeta\right)  ^{2}dA\\
& \leq C\left\vert V_{G}^{\beta_{1}}\right\vert ^{\varepsilon\left(
\gamma-\beta_{1}\right)  }\int_{V_{G}^{\beta_{1}}}\left(  \int_{\mathbb{D}%
\setminus V_{G}^{\beta}}\frac{\left\vert b^{\prime}\left(  \zeta\right)
\right\vert \left(  1-\left\vert \zeta\right\vert \right)  ^{s}}{\left\vert
1-\overline{\zeta}z\right\vert ^{2+s}}d\zeta\right)  ^{2}dA\\
& \leq C\left\vert V_{G}^{\beta_{1}}\right\vert ^{\varepsilon\left(
\gamma-\beta_{1}\right)  }\left\Vert b\right\Vert _{\mathcal{D}}^{2}\leq
C\left\Vert T_{b}\right\Vert ^{2}Cap_{\mathcal{T}}\left(  E,F\right)  .
\end{align*}

We continue from (\ref{rr4}$).$ We know that $\left\vert (4)\right\vert
\leq\sqrt{(3_{A})/\varepsilon}\sqrt{(4_{A})}$ We estimate $(3_{A})$ using
(\ref{VII}) and $(4_{A})$ using (\ref{pres}). After that we continue by using
(\ref{inpart});
\begin{align}
\left\vert (4)\right\vert  & \leq\sqrt{C\mu_{b}\left(  V_{G}\right)
+C\left\Vert T_{b}\right\Vert ^{2}Cap_{\mathcal{T}}\left(  E,F\right)
}\label{I'}\\
& \text{ }\times\sqrt{C\mu_{b}\left(  V_{G}^{\beta}\setminus V_{G}\right)
+C\left\Vert T_{b}\right\Vert ^{2}Cap_{\mathcal{T}}\left(  E,F\right)
}\nonumber
\end{align}%
\begin{align*}
& \leq\sqrt{C\mu_{b}\left(  V_{G}\right)  +C\left\Vert T_{b}\right\Vert
^{2}Cap_{\mathcal{T}}\left(  E,F\right)  }\\
& \text{ }\times\sqrt{\varepsilon\mu_{b}\left(  V_{G}\right)  +C\left\Vert
T_{b}\right\Vert ^{2}Cap_{\mathcal{T}}\left(  E,F\right)  }\\
& \leq\sqrt{\varepsilon}\mu_{b}\left(  V_{G}\right)  +C\sqrt{\mu_{b}\left(
V_{G}\right)  }\sqrt{\left\Vert T_{b}\right\Vert ^{2}Cap_{\mathcal{T}}\left(
E,F\right)  }\\
& \text{ }+C\left\Vert T_{b}\right\Vert ^{2}Cap_{\mathcal{T}}\left(
E,F\right)  .
\end{align*}
Now, recalling that $f^{\prime}=b^{\prime}+\Lambda b^{\prime},$%
\begin{align}
\left\Vert \Phi f\right\Vert _{\mathcal{D}}^{2}  & \leq C\int\left\vert
\Phi^{\prime}\left(  z\right)  f\left(  z\right)  \right\vert ^{2}%
dA+C\int\left\vert \Phi\left(  z\right)  (b^{\prime}(z)+\Lambda b^{\prime
}(z))\right\vert ^{2}dA\label{rr6}\\
& \leq C\left(  3_{B}\right)  +C\frac{1}{\varepsilon}\left(  3_{A}\right)
+C\left(  4_{A}\right)  .\nonumber\\
& \leq C\mu_{b}\left(  V_{G}\right)  +C\left\Vert T_{b}\right\Vert
^{2}Cap_{\mathcal{T}}\left(  E,F\right)  ,\nonumber
\end{align}
by (\ref{pres}) and the estimates (\ref{VII}) and (\ref{VIII}) for $(3_{A})$
and $(3_{B})$.
\end{proof}
\end{lemma}

Using Proposition \ref{3} and the estimates (\ref{II}), (\ref{I'}) and
(\ref{rr6}) in (\ref{rr3}) we obtain%
\begin{align*}
\mu_{b}\left(  V_{G}\right)   & \leq\sqrt{\varepsilon}\mu_{b}\left(
V_{G}\right)  +C\left\Vert T_{b}\right\Vert ^{2}Cap_{\mathcal{T}}\left(
E,F\right) \\
& +C\sqrt{\left\Vert T_{b}\right\Vert ^{2}Cap_{\mathcal{T}}\left(  E,F\right)
}\sqrt{\mu_{b}\left(  V_{G}\right)  }\\
& \leq\sqrt{\varepsilon}\mu_{b}\left(  V_{G}\right)  +C\left\Vert
T_{b}\right\Vert ^{2}Cap_{\mathcal{T}}\left(  E,F\right)  .
\end{align*}
We absorb the first term into the right side. Now using Lemma
\ref{newcondenser}, Lemma \ref{contain} again, and Corollary \ref{capequiv} we
obtain
\[
Cap_{\mathcal{T}}\left(  E,F\right)  \leq CCap_{\mathbb{D}}G.
\]
Finally we have%
\[
\mu_{b}\left(  V_{G}\right)  \leq C\left\Vert T_{b}\right\Vert ^{2}%
Cap_{\mathcal{T}}\left(  E,F\right)  \leq C\left\Vert T_{b}\right\Vert
^{2}Cap_{\mathbb{D}}G,
\]
which is (\ref{maintheta}).

\section{Appendix on Tree Extremals}

Let $E$ be a stopping time in $\mathcal{T}$. Recall that
\begin{equation}
Cap_{\mathcal{T}}\left(  E\right)  =\inf\{\Vert h\Vert_{\ell^{2}}^{2}%
:\ Ih\geq1\ \text{on}\ E\}.\label{capfunction}%
\end{equation}
We call functions which can be used in computing the infimum
\textit{admissible.\textrm{\ }}

Much of the following proposition as well as Proposition \ref{extcondenser}
could be extracted from general capacity theory such as presented in, for
instance, \cite{AH}. Statement (3) is the discrete analog of the fact that
continuous capacity can be interpreted as the derivative at infinity of a
Green function.

\begin{proposition}
\label{extremalfunction} Suppose $E\subset\mathcal{T}$ is given.

\begin{enumerate}
\item There is a function $h$ such that the infimum in the definition of
$\mbox{Cap}_{\mathcal{T}}(E)$ is achieved.

\item If $x\notin E$,
\begin{equation}
h(x)=h(x_{+})+h(x_{-}).\label{sum}%
\end{equation}

\item $h(o)=\Vert h\Vert_{\ell^{2}}^{2}$.

\item $h$ is strictly positive on $\mathcal{G}\left(  o,E\right)  $ and zero elsewhere.

\item $Ih|_{E}=1.$
\end{enumerate}
\end{proposition}

\begin{proof}
Consider first the case when $E$ is a finite subset of $\mathcal{T}$.
Multiplying an admissible function by the characteristic function of
$\mathcal{G}\left(  o,E\right)  $ leaves it admissible and reduces the
$\ell^{2}$ norm. Hence we need only consider functions supported on the finite
set of vertices in $\mathcal{G}\left(  o,E\right)  .$ In that context it is
easy to see that an extremal exists, call it $h.$ Now consider (2). Suppose
$x\in\mathcal{T}\ $%
$\backslash$%
$E$ and consider the competing function $h^{\ast}$ which takes the same values
as $h$ except possible at $x,x_{+},$ and $x_{-}$ and whose values at those
points are determined by
\begin{align*}
& \left(  1\right)  .\text{ }h^{\ast}(x)+h^{\ast}(x_{+})=h(x)+h(x_{+})\text{
and }h^{\ast}(x)+h^{\ast}(x_{-})=h(x)+h(x_{-})\\
& \left(  2\right)  .\text{ }h^{\ast}(x)^{2}+h^{\ast}(x_{+})^{2}+h^{\ast
}(x_{-})^{2}\text{ is minimal subject to }\left(  1\right)  .
\end{align*}
Then $h^{\ast}$ is admissible, $\Vert h^{\ast}\Vert_{\ell^{2}}^{2}\leq\Vert
h\Vert_{\ell^{2}}^{2},$ and, doing the calculus problem, $h^{\ast}$ satisfies
(\ref{sum}). Hence $h$ must satisfy (\ref{sum}).

If $h(x)<0$ at some point, replacing its value by zero leaves the function
admissible while reducing the $\ell^{2}$ norm, hence $h\geq0$. To complete the
proof of (4) we must show that we cannot have an $x\in\mathcal{G}\left(
o,E\right)  $ at which $h(x)=0.$ Suppose we had such a point. By (\ref{sum})
and the fact that $h\geq0$, we have $h\equiv0$ on $S_{\mathcal{T}}(x)$. Hence
by admissibility $Ih(x^{-1})\geq1$. Let $y\neq x$ be the point such that
$x^{-1}=y^{-1}$. If $h(y)>0$ then setting $h(y)=0$ we would decrease the
$\ell^{2}$ norm while keeping the function admissible. Thus $h(y)=0$ and, by
(\ref{sum}), $h(x^{-1})=0$. Continuing in this way we find that $h\equiv0$ an
the geodesic from $o$ to some $e\in E,$ an impossibility for an admissible
function. Item (5) is a consequence of this. If $Ih(e)>1$ for some $e\in$ $E$
and $h(e)>0$ then we could decrease $h(e)$ slightly, reducing the norm of $h$
and still have $h$ admissible; contradicting the supposition that $h$ is extremal.

It remains to show (3) and we do that by induction on the size of $E$. If
$E=\{e\}$ is a single point having distance $d-1\geq0$ from $o$ then the
extremal is $h\equiv1/d$ on $[o,e]$ and $\Vert h\Vert_{\ell^{2}}%
^{2}=d(1/d)^{2}=h(o)$. Given $E$ with more than one point, let $z$ be the
uniquely determined branching point in $\mathcal{G}\left(  o,E\right)  $
having the least distance from the root. Consider the rooted trees
$\mathcal{T}_{\pm}=S(z_{\pm})$ with roots $z_{\pm}.$ Set $E_{\pm}%
=E\cap\mathcal{T}_{\pm}$ and let $h_{\pm}$ be the extremal functions for the
computation of $Cap_{\mathcal{T}_{\pm}}(E_{\pm})$. By induction, we have that
$\Vert h_{\pm}\Vert_{\ell^{2}}^{2}=h_{\pm}(z_{\pm})$. From properties (1)-(5)
satisfied by the extremal functions $h$, $h_{+}$ and $h$ it is easy to see
that
\[
h(x)=\left\{
\begin{array}
[c]{c}%
(1-Ih(z))h_{\pm}(x)\ \text{if}\ x\in{\mathcal{G}}(z_{\pm})\\
h(o)\text{ if}\ x\in\lbrack o,z]\text{ }\\
0\text{ otherwise. }%
\end{array}
\right.
\]
In particular, $Ih(z)=dh(o)$ if there are $d$ points in $[o,z]$, so that
\begin{equation}
h(o)=h(z)=h(z_{+})+h(z_{-})=\frac{h_{+}(z_{+})+h_{-}(z_{-})}{1-Ih(z)}%
=\frac{h_{+}(z_{+})+h_{-}(z_{-})}{1-dh(o)}.\label{wallstreet}%
\end{equation}
Rescaling and using the induction hypothesis,
\begin{align*}
\Vert h\Vert_{\ell^{2}}^{2}  & =\left(  \Vert h_{+}\Vert_{\ell^{2}}^{2}+\Vert
h_{-}\Vert_{\ell^{2}}^{2}\right)  (1-dh(o))^{2}+dh(o)^{2}\\
& =\left(  h_{+}(z_{+})+h_{-}(z_{-})\right)  (1-dh(o))^{2}+dh(o)^{2}\\
& =\frac{h(z_{+})+h(z_{-})}{1-dh(o)}(1-dh(o))^{2}+dh(o)^{2}%
\end{align*}%
\begin{align*}
& =\frac{h(z)}{1-dh(o)}(1-dh(o))^{2}+dh(o)^{2}\\
& =\frac{h(o)}{1-dh(o)}(1-dh(o))^{2}+dh(o)^{2}\\
& =h(o).
\end{align*}
We note in passing that, by (3), formula (\ref{wallstreet}) gives a recursive
formula for computing tree capacities.

Suppose now that $E$ is infinite. Select a sequence of finite sets
$E_{n}=\{e_{1},\dots,e_{n}\}$ such that $E_{n}\nearrow E$. Let $h_{n}$ be the
corresponding extremal functions and $H_{n}=Ih_{n}$. We claim that the
sequence $H_{n}$ increases, in the sense specified below. Let $K=H_{n}%
-H_{n-1}=I(h_{n}-h_{n-1})=Ik_{n}$. By (\ref{sum}), the function $K$ satisfies
the mean value property on ${\mathcal{G}}(o,E_{n})\setminus(\{o\}\cup E_{n}%
)$:
\[
K(x)=\frac{1}{3}[K(x_{+})+K(x_{-})+K(x^{-1})],\ \mbox{if}\ x\in{\mathcal{G}%
}(o,E_{n})\setminus(\{o\}\cup E_{n}).
\]
Moreover, $K$ vanishes on $\{o\}\cup E_{n-1}$ and it is positive at $e_{n}$,
since $H_{n-1}(e_{n})\leq1=H_{n}(e_{n})$, by (3) and (4). By the maximum
principle (an easy consequence of the mean value property), $K_{n}\geq0$ in
${\mathcal{G}}(o,E_{n})$. Hence, the limit $Ih=H=\lim_{n}H_{n}$ exists in
${\mathcal{G}}(o,E)$ and it is finite because each $H_{n}$ is bounded above by
$1$. Since $h(x)=H(x)-H(x^{-1})=\lim h_{n}(x)$, $h$ is admissible for $E$ and
it satisfies (3), (4) and (5).

Also, $h_{n}\rightarrow h$ as $n\rightarrow\infty$, pointwise, and $\Vert
h_{n}\Vert_{\ell^{2}}^{2}=h_{n}(o)\rightarrow h(o)$, by dominated convergence,
hence,
\[
h(o)=\lim_{n\rightarrow\infty}\Vert h_{n}\Vert_{\ell^{2}}^{2}=\Vert
h\Vert_{\ell^{2}}^{2},
\]
which is (3) for $h$.

It remains to prove that $h$ is extremal. Suppose $k$ is another admissible
function for $E$, and let $k_{n}$ be its restriction to ${\mathcal{G}}%
(o,E_{n})$, which is clearly admissible for $E_{n}$. By the extremal character
of the functions $h_{n}$, we have
\[
\Vert k\Vert_{\ell^{2}}^{2}=\lim_{n\rightarrow\infty}\Vert k_{n}\Vert
_{\ell^{2}}^{2}\leq\lim_{n\rightarrow\infty}\Vert h_{n}\Vert_{\ell^{2}}^{2}=\lim_{n\rightarrow\infty}h_{n}(o)=h(o)=\Vert h\Vert_{\ell^{2}}^{2},
\]
hence, $h$ is extremal among the admissible functions for $E$%
.\textit{\textrm{\ }}
\end{proof}

\begin{proof}
[Proof of Proposition \ref{extcondenser}]Consider each $e\in E$ as the root of
the tree $\mathcal{T}_{e}=S(e).$ Set $F_{e}=F\cap S(e)$ and let $h_{e}$ be the
extremal function (from the previous proposition) for computing
$Cap_{\mathcal{T}_{e}}(F_{e}).$ Using the previous proposition it is
straightforward to check that $h=\sum h_{e}$ is the required extremal function
and has the required properties.
\end{proof}

\section{Acknowledgements}

This work was begun while the second, third, and fourth author were visiting
the Fields Institute. We thank the institute for its hospitality and its
excellent working conditions. While there we had interesting discussions with
Michael Lacey. He suggested an approach to the theorem that was quite
different from what we had been envisioning and his comments helped shape our
approach. In particular they led us to suspect that a good estimate for the
"collar error term", Proposition \ref{3}, could play a crucial role. We thank
him for his involvement.

\end{document}